\documentclass[11pt,oneside,leqno]{amsart}
\usepackage{amsxtra}
\usepackage{amsopn}
\usepackage{amsmath,amsthm,amssymb}
\usepackage{amscd}
\usepackage{amsfonts}
\usepackage{latexsym}
\usepackage{verbatim}
\usepackage{pb-diagram}
\usepackage{color}

\theoremstyle{plain}
\newtheorem{theorem}{Theorem}[section]
\newtheorem*{theorem*}{Theorem}
\newtheorem{definition}[theorem]{Definition}
\newtheorem{lemma}[theorem]{Lemma}
\newtheorem{prop}[theorem]{Proposition}
\newtheorem{cor}[theorem]{Corollary}
\newtheorem{rem}[theorem]{Remark}

\newtheorem*{mt*}{Main Theorem}
\newcommand\C{{\mathbb C}}

\newcommand\h{{harm}}

\newcommand\R{{\mathbb R}}

\newcommand\Span{{\hbox{\rm Span}}}

\newcommand\Ker{\hbox{Ker}\,}
\newcommand\Image{\hbox{Im}\,}

\newcommand\g{{\mathfrak{g}}}

\newcommand\id{{\hbox{\em id}}}
\begin{document}
\title[Almost K\"ahler structures on four dimensional unimodular Lie algebras]
{Almost K\"ahler structures on four dimensional unimodular Lie algebras}
\author{Tian-Jun Li and Adriano Tomassini}
\date{\today}
\address{Tian-Jun Li \\ School of Mathematics\\
University of Minnesota \\
Minneapolis\\ MN 55455, USA} \email{tjli@math.umn.edu}
\address{Adriano Tomassini, Dipartimento di Matematica\\ Universit\`a di Parma\\
Parco Area delle Scienze 53/A\\
43124 Parma\\ Italy} \email{adriano.tomassini@unipr.it}
\thanks{The second author was supported by GNSAGA
of INdAM}
\begin{abstract}
Let $J$ be an almost complex structure on a $4$-dimensional and unimodular Lie algebra $\mathfrak{g}$.
We show that
there exists a symplectic form taming $J$ if and only if there is a symplectic form compatible with $J$.
We also introduce groups $H^+_J(\mathfrak{g})$ and
$H^-_J(\mathfrak{g})$ as the subgroups of the Chevalley-Eilenberg cohomology classes which can
be represented by $J$-invariant, respectively $J$-anti-invariant, $2$-forms on $\mathfrak{g}$.
and we prove a cohomological $J-$decomposition theorem following \cite{DLZ}:
$H^2(\mathfrak{g})=H^+_J(\mathfrak{g})\oplus H^-_J(\mathfrak{g})$. We discover that tameness of $J$ can be characterized
in terms of the dimension of $H^{\pm}_J(\mathfrak{g})$, just as in the complex surface case. We also describe
the tamed and compatible symplectic cones respectively. Finally, two applications to homogeneous $J$ on $4-$manifolds are obtained.
\end{abstract}
\maketitle
\tableofcontents
\section*{Introduction}
In this paper we are interested in the geometry of almost complex
structures on four dimensional Lie algebras and their cohomological properties, especially from the symplectic point of view.

Given a $2n$-dimensional
oriented real Lie algebra $\mathfrak{g}$, let $\Lambda^k(\mathfrak{g})$ and $\Lambda^r(\mathfrak{g}^*)$ denote the spaces of $k$-vectors,
$r$-forms respectively on $\mathfrak{g}$.
Let $H^{*}(\g)$ be the Chevalley-Eilenberg cohomology group of $\mathfrak{g}$. Then $\mathfrak{g}$ is called {\em unimodular} if its top dimensional cohomology vanishes.

The special feature in dimension four are, up to a choice of a volume form, the symmetric wedge product pairing on $\Lambda^2(\mathfrak{g}^*)$ (as well as the one on $\Lambda^2(\mathfrak{g})$) with signature $(3, 3)$,
and the induced one on $H^2(\mathfrak{g})$ with signature $(b^+(\mathfrak{g}),b^+(\mathfrak{g}) )$ when $\mathfrak{g}$ is unimodular.
Here $b^+(\mathfrak{g})$ is the maximal dimension of a definite subspace of $H^2(\mathfrak{g})$.

Let $J$ be an almost complex structure on $\mathfrak {g}$, i.e. an automorphism of $\mathfrak{g}$ with $J^2=-\id$, such that the natural
orientation induced by $J$ is the same as the fixed one.
\newline Consider the involution
 on the space of
$2$-forms $\Lambda^2(\mathfrak{g}^*)$, $\alpha(\cdot, \cdot)
\rightarrow \alpha(J\cdot, J\cdot)$. Accordingly, we have the $\pm 1-$eigenspace decomposition:
\begin{equation} \label{formtype-Lie-algebras}
\Lambda^2(\mathfrak{g}^*)=\Lambda_J^+(\mathfrak{g}^*)\oplus \Lambda_J^-(\mathfrak{g}^*).
\end{equation}
We recall the following
\begin{definition}\label{tc} Let $\mathfrak{g}$ be endowed with an almost complex structure $J$. A closed $2$-form $\varphi$ on $\mathfrak{g}$ is said to be {\em $J-$tamed} if it is positive on $v\wedge Jv$ for
any nonzero $v\in \mathfrak{g}$. A $J-$tamed $\varphi$ is said to be {\em $J-$compatible} if $\varphi\in \Lambda^+_J(\mathfrak{g}^*)$. The almost complex structure $J$ is said to be {\em tamed} if there exists a $J-$tamed form and {\em almost K\"ahler}, in short, $\hbox{\rm AK}$, if there exists a $J-$compatible form.
\end{definition}
Our first main result is (see Theorem \ref{compatible}):

\begin{theorem} \label{tamed-and-compatible-intro} Let $\mathfrak{g}$ be a $4-$dimensional Lie algebra with $B\wedge B=0$, where $B\subset \Lambda^2(\mathfrak{g})$ is the space of boundary $2-$vectors.
Then an almost complex
structure $J$ is tamed if and only if $J$ is almost K\"ahler.
\end{theorem}

If $\mathfrak{g}$ is $4-$dimensional and unimodular, then the condition $B\wedge B=0$ is satisfied.
In particular, Theorem \ref{tamed-and-compatible-intro} confirms the tame/compatible
Question of Donaldson (see \cite[Question 2]{D}) in the Lie algebra setting.

We further discover that tamed almost complex structures can be characterized cohomologically.
For this purpose, introduce the following subgroups $H_J^{\pm}(\g)$ of $H^2(\g)$.
\begin{definition}\label{invariant-antiinvariant}
Let $ \mathcal Z^2$ denote the space of closed $2$-forms on $\mathfrak{g}$ and let
$\mathcal Z_J^{\pm} = \mathcal Z^2 \cap \Lambda_J^{\pm}(\mathfrak{g}^*)$. Define
\begin{equation}
H_J^{\pm}(\mathfrak{g})=\{ \mathfrak{a} \in H^2(\mathfrak{g})\,\,\, |\,\,\, \exists \; \alpha\in \mathcal
Z_J^{\pm} \mbox{ such that } [\alpha] = \mathfrak{a} \}.
\end{equation}
\end{definition}
If $\mathfrak{g}$ is $4$-dimensional
and unimodular, we establish a $J-$decomposition result for
 $H^2(\mathfrak{g})$, i.e.
$$
H^2(\mathfrak{g})=H^+_J(\mathfrak{g})\oplus H^-_J(\mathfrak{g})
$$
(see Theorem
\ref{Nomizu-1-Lie-algebra}).

The following is our second main result, which is a characterization of tamed almost
complex structures in terms of the dimension $h^{-}$ of $H^{-}_J(\mathfrak{g})$ (see Theorem \ref{tame-new}):
\begin{theorem} \label{tame-char} Suppose $J$ is an almost complex structure on a 4-dimensional unimodular Lie algebra $\mathfrak{g}$. Then
$h_J^-=b^+(\mathfrak{g})$ or $b^+(\mathfrak{g})-1$.
Moreover, $J$ is tamed if and only if
$$
h^-_J=b^+(\mathfrak{g})-1.
$$
\end{theorem}

A particular interesting consequence (see Theorem \ref{tame-new}) is that, when $b^+(\mathfrak{g})=2$, an
almost complex structure is either integrable or almost K\"ahler.

We also introduce tamed and compatible cones, and we are able to describe both of them when the Lie algebra is $4-$dimensional and unimodular.
\begin{definition}
The convex cones
$$
\mathcal{K}^t_J=\left\{[\omega]\in H^2(\mathfrak{g})\,\,\,\vert\,\,\,\omega\,\,
\hbox{\rm is}\,\, J\hbox{-\rm tamed}\right\} \quad \hbox{
and}\quad
\mathcal{K}^c_J=\left\{[\omega]\in H_J^+(\mathfrak{g})\,\,\,\vert\,\,\,\omega\,\,
\hbox{\rm is}\,\, J\hbox{-\rm compatible}\right\}\,
$$
are called the {\em $J-$tamed symplectic cone} and the {\em $J-$compatible symplectic cone} respectively.
\end{definition}
We show the following (see Theorem \ref{AK-tamed-cones})
\begin{theorem}\label{AK-tamed-cones-intro}
For an almost complex
structure $J$ on a
$4$-dimensional unimodular Lie algebra, the $J-$compatible cone $\mathcal K_J^c(\mathfrak{g})$ is a connected component of $\{e\in H_J^+(\mathfrak{g})|e^2>0\}$,
and the $J-$tamed cone $\mathcal K_J^t(\mathfrak{g})=\mathcal K_J^c(\mathfrak{g})+H_J^-(\mathfrak{g})$.
\end{theorem}
We want to emphasize that although there is a classification of unimodular symplectic 4-dimensional Lie algebras (\cite{Ov}),
our proofs of Theorems \ref{tame-char} and \ref{AK-tamed-cones-intro} do not rely on it.
We systematically explore the pairings on $\Lambda^2$, and it seems that much can be extended to certain class of higher dimensional unimodular symplectic Lie algebras.

We finally note that the results above obtained at the Lie algebra level have analogues for a left-invariant almost complex structure $J$ on a
 quotient $M=\Gamma\backslash G$ of a Lie group $G$ by a
discrete subgroup $\Gamma\subset G$. \newline

The paper is organized as follows: in Section \ref{AKSLA} we start by describing carefully the natural non-degenerate bilinear pairings on $\Lambda^k(\mathfrak{g})$ and their properties and by recalling the basic definitions of almost Hermitian geometry on Lie algebras. Sections \ref{tamed-versus-compatible} and \ref{unimodular4-dimensional} are devoted to the proofs of the main results, namely Theorems \ref{tamed-and-compatible-intro}, \ref{tame-char} and \ref{AK-tamed-cones-intro}, which are included in Theorems \ref{compatible}, \ref{tame-new} and \ref{AK-tamed-cones} respectively. Key tools in their proofs are a finite dimensional version of the Hahn-Banach theorem together with the results by \cite[Theorem 3.2]{sullivan} and \cite[Theorem 14]{HL}.
In particular, in Section \ref{tamed-versus-compatible} we prove Theorem \ref{compatible}.
In Section \ref{unimodular4-dimensional}, we focus on unimodular $4$-dimensional Lie algebras endowed with an almost complex structure $J$.
We prove Theorem \ref{Nomizu-1-Lie-algebra} and, as a consequence, we observe that there is a homological $J-$decomposition on $\mathfrak{g}$
(see Remark \ref{homology decomposition}).
As an application of the Theorem \ref{tame-char}, we construct a $2$-parameter family of almost complex structures on the Lie algebras $\mathfrak{nil}^3\times \R$ and $\mathfrak{nil}^4$ (see subsection \ref{examples}) which cannot be tamed by any symplectic form.
In Section 4 we establish the analogous results for homogeneous $J$ on closed $4-$manifolds.
\smallskip

\noindent{\sl Acknowledgments.} The first author would like to thank the Department of Mathematics at University of Parma and the Mathematical Science Center at
Tsinghua University, and the second author would like to thank the School of Mathematics at University of Minnesota,
 for their warm hospitality. We are also grateful to Daniele Angella, Tedi Dr\v{a}ghici and Weiyi Zhang for useful discussions.

\section{Almost complex geometry on Lie algebras}\label{AKSLA}
We start by presenting some preliminaries and fixing some notations.

Let $\mathfrak{g}$ be a $2n$-dimensional real Lie algebra, which we will assume to be oriented.
The Lie bracket induces the operator $d:\Lambda^1(\mathfrak{g}^*)\to \Lambda^2(\mathfrak{g}^*)$ by $d\phi(u,v)=-\phi([u,v])$,
which makes $\Lambda^*(\mathfrak{g}^*)$ a differential algebra
by the Leibniz rule. This is the so called {\em Chevalley-Eilenberg complex} of the Lie algebra $\mathfrak{g}$.

We will denote by $\mathcal{Z}^r$, $\mathcal{B}^r$, the space of closed, exact $r-$forms respectively on $\mathfrak{g}$.
Then the $r-$th Chevalley-Eilenberg cohomology is
$
H^r(\mathfrak{g})=\mathcal{Z}^r/\mathcal{B}^r
$.

Let us also define the homology $H_*(\mathfrak{g})$ of $\mathfrak{g}$.
Given any $k$-vector
$u\in\Lambda^k(\mathfrak{g})$, set
$$
T_u(\varphi)=\varphi(u),
$$
for every $\varphi\in\Lambda^k(\mathfrak{g}^*)$. Then $T_u$ is a linear functional on
$\Lambda^k(\mathfrak{g}^*)$. For every $w\in\Lambda^k(\mathfrak{g})$, set
$$
dT_w(\varphi)=(-1)^{k+1}T_w(d\varphi),
$$
for every $\varphi\in\Lambda^{k-1}(\mathfrak{g}^*)$. Then, $dT_w$ is a linear functional on
$\Lambda^{k-1}(\mathfrak{g}^*)$. A $k$-vector $u$ is a {\em k-cycle} if $dT_u=0$, a {\em k-boundary} if there exists a $(k+1)$-vector
$w$ such that $T_u=dT_w$. Now the $k-$th homology of $\mathfrak{g}$ is the quotient of the space of $k-$cycles by the space of $k-$boundaries. In the sequel, we will
denote by $Z^k$, $B^k$, the space of $k$-cycles, $k$-boundaries on $\mathfrak{g}$ respectively.
\subsection{Bilinear pairings}\label{bilinearpairings}
We will consider a number of non-degenerate bilinear pairings including
$$
\Psi^k:\Lambda^k(\mathfrak{g})\times\Lambda^{k}(\mathfrak{g^*})\to\R, \quad \Psi^k(u, \alpha)=\alpha(u),
$$
for each $k$.
Let us first state some general properties of a non-degenerate bilinear pairing on finite dimensional vector spaces.

Let $V, W$ be real vector spaces and $ \Gamma:V\times W\to \R$ a bilinear pairing. Let $V_0\subset V$ and $W_0\subset W$ be subspaces and set
$$
V_0^{\bot}=\{w\in W\,\,\,|\,\,\,\Gamma(v, w)=0, \hbox{ for any } v\in V_0\},
$$
$$
W_0^{\bot}=\{v\in V\,\,\,|\,\,\,\Gamma(v, w)=0, \hbox{ for any } w\in W_0\}.
$$
We say that $V_0\subset V$ and $W_0\subset W$ are a $\Gamma-${\em complementary pair}, if
$$
V_0=W_0^{\bot}\,\,\,\hbox{ and }\,\,\,W_0=V_0^{\bot}.
$$
\begin{lemma}\label{complement}
Suppose $V$ and $W$ have finite dimension and $\Gamma:V\times W\to \R$ is a non-degenerate pairing. Then
\begin{itemize}
\item $\dim V=\dim W$;
\item $\dim W_0 +\dim W_0^{\bot} =\dim W$ for any $W_0\subset W$;
\item For any $V_0\subset V$ and $W_0\subset W$, if $V_0=W_0^{\bot}$ then $W_0=V_0^{\bot}$;
\item If $V_1\subset V_0$, $W_0\subset W_1$, $(V_0, W_0)$ and $(V_1, W_1)$ are both $\Gamma-$complementary pair, then there is an induced pairing
$\bar{\Gamma}:V_0/V_1\times W_1/W_0\to \R$, which is also non-degenerate.
\end{itemize}
\end{lemma}
This is straightforward since $\dim V$ and $\dim W$ are finite.
\begin{lemma}\label{involution}
Suppose $V$ and $W$ have finite dimension and $\Gamma:V\times W\to \R$ is a non-degenerate pairing.
Suppose that there are involutions $\iota_V, \iota_W$ on $V$ and $W$ such that
\begin{equation}\label{involution-invariance}
\Gamma(\iota_V(\cdot), \iota_W(\cdot))=\Gamma(\cdot, \cdot).
\end{equation}
Let $V^{\pm}$ and $W^{\pm}$ be the $\pm-$eigenspaces of $\iota_V$ and $\iota_W$ respectively, and
$$\Gamma_{\pm}:V^{\pm}\times W^{\pm}\to \R$$ be the restriction of $\Gamma$. Then
\begin{itemize}
\item $\Gamma(V^{\pm}, W^{\mp})=0$;
\item Furthermore, $V^{\pm}$ and $W^{\mp}$ are a $\Gamma-$complementary pair;
\item $\Gamma_{\pm}$ is non-degenerate.
\end{itemize}
\end{lemma}
\begin{proof}
Suppose $v\in V^{\pm}$ and
$w\in W^{\mp}$. Then by \eqref{involution-invariance},
$$
\Gamma(v, w)=\Gamma(\iota_V(v), \iota_W(w))=-\Gamma(v,w).
$$
Thus $\Gamma(v, w)=0$.

For the second bullet we just show that $V^+$ and $W^-$ are a $\Gamma-$complementary pair, the other case is similar.
Since $\Gamma$ is non-degenerate, by the third bullet of Lemma \ref{complement} it is sufficient to prove that
$V^+=(W^{-})^\bot.$
By the first bullet we have $V^+\subset(W^{-})^\bot.$

Conversely, let $v\in (W^{-})^\bot$. Write
 $v=v_++v_-$, where $v_{\pm}\in V^{\pm}$. Given any
$w\in W^-$, we have:
$$
0=\Gamma(v, w)=\Gamma(v_+ + v_-,w)=\Gamma(v_-,w).
$$
Therefore, $\Gamma(v_-, W^+\oplus W^-)=0$. Consequently, $v_-=0$ since $\Gamma$ is non-degenerate. Hence
$v=v_+\in V^+$ and $ V^+\supset (W^{-})^\bot$.\vskip.2truecm\noindent
\smallskip

The third bullet is a direct consequence of the second bullet and the non-degeneracy of $\Gamma$.
\end{proof}
Applying Lemma \ref{complement} we obtain
\begin{lemma}\label{closed-exact'} $(Z_k, \mathcal B^k)$ is a $\Psi^k-$complementary pair, and so is
$(B_k, \mathcal Z^k)$.
\noindent In particular, there is an induced pairing
$$
\bar \Psi^k:H_k(\mathfrak{g})\times H^k(\mathfrak{g})\to \R,
$$
which is also non-degenerate.
\end{lemma}
\begin{proof}
By the fourth bullet of Lemma \ref{complement}, the last statement follows from the first statement. For the first statement we just treat the first pair, the second pair is similar.
Further, by the third bullet of Lemma \ref{complement} we just need to
 show that $Z_k=\mathcal (\mathcal{B}^{k})^\bot$.

$(\subset)$ Let $u\in Z_k$. Then, for every $\gamma\in\Lambda^{k-1}(\mathfrak{g}^*)$,
$$
\Psi^k(u,d\gamma)=T_u(d\gamma)=(-1)^{k+1}dT_u(\gamma)=0,
$$
i.e., $u\in (\mathcal{B}^{k})^\bot$.

$(\supset)$ Let $u\in(\mathcal{B}^{k})^\bot$. Then,
$$
0=\Psi^k(u,d\gamma)=T_u(d\gamma)=(-1)^{k+1}dT_u(\gamma)
$$
for every $\gamma\in\Lambda^{k-1}(\mathfrak{g}^*)$. Hence, $dT_u=0$ and $u\in Z_k$. \newline
\end{proof}
\smallskip

We also need to consider the following pairings.
Fix $\zeta\neq 0 \in\Lambda^{2n}(\mathfrak{g})$ inducing the given orientation of $\mathfrak{g}$, and $\eta\in\Lambda^{2n}(\mathfrak{g}^*)$ with $\Psi^{2n}(\zeta, \eta)=1$.
\begin{definition}\label{def-pairing}
For each $k$, we define the bilinear pairing
$$
\Phi_{\zeta}^k: \Lambda^k(\mathfrak{g}^*)\times\Lambda^{2n-k}(\mathfrak{g}^*)\to\R,
$$
\begin{equation}\label{zeta-pairing}
\Phi_{\zeta}^k(\alpha,\beta)=(\alpha\wedge\beta)(\zeta),
\end{equation}
for every $\alpha\in\Lambda^k(\mathfrak{g}^*)$, $\beta\in\Lambda^{2n-k}(\mathfrak{g}^*)$, and similarly define
$$\Phi_{\eta}^k:\Lambda^k(\mathfrak{g})\times\Lambda^{2n-k}(\mathfrak{g})\to\R, \quad \quad \Phi_{\eta}^k(u,w)=\eta (u\wedge w).$$
\end{definition}

A $k-$vector in $\Lambda^k(\mathfrak{g})$ is called {\em simple} if it is of the form $w_1\wedge \ldots\wedge w_k$ for $w_i\in \mathfrak{g}$.
{\em Simple forms} are defined similarly. Geometrically, a non-zero simple $k-$vector $w_1\wedge\ldots\wedge w_k$ gives rise to a $k-$plane,
$\Span\{w_1,\ldots ,w_k\}$. In fact, $v\in \mathfrak{g}$ belongs to $\Span\{w_1,\ldots ,w_k\}$ if and only if $v\wedge (w_1\wedge\ldots \wedge w_k)=0$.
If $\alpha\in\Lambda^k(\mathfrak{g}^*)$ and $\{v_1,\ldots,v_k\}\subset\mathfrak{g}$, our convention is that
$\alpha(v_1\wedge \cdots \wedge v_k)=\alpha(v_1,\ldots ,v_k)$.

If we choose a basis $\{w_i\}$ of $\mathfrak{g}$ and the dual basis $\{w^i\}$ of $\mathfrak{g}^*$, for each $k$,
there are associated bases of $\Lambda^k(\mathfrak{g})$ and $\Lambda^k(\mathfrak{g}^*)$ consisting of simple $k-$vectors and $k-$forms. By considering such bases, it is easy to prove

\begin{lemma}\label{pairing}
1. $\Phi^k_{\zeta}$ and $\Phi_{\eta}^k$ are non-degenerate pairings for each $k$.

2. The linear maps $$
G^k_\eta:\Lambda^k(\mathfrak{g})\to\Lambda^{2n-k}(\mathfrak{g}^*)\,, \quad G^{2n-k}_\zeta:\Lambda^{2n-k}(\mathfrak{g}^*)\to\Lambda^{k}(\mathfrak{g})
$$
defined by
$$
G^k_\eta(u)(w)=\Phi_{\eta}^k(u,w)=\eta (u\wedge w),\,\quad G^{2n-k}_\zeta(\alpha)(\beta)=\Phi_{\zeta}^{2n-k}(\alpha, \beta)= (\alpha\wedge \beta)(\zeta)
$$
respectively, are inverses of each other.


3. For $\alpha\in \Lambda^k(\mathfrak{g}^*)$, $\beta\in \Lambda^{2n-k}(\mathfrak{g}^*)$,
\begin{equation}\label{gg}
T_{G_{\zeta}^k(\alpha)}(\beta)=T_{\zeta}(\beta\wedge \alpha).
\end{equation}

4. For $\alpha\in\Lambda^{2n-k}(\mathfrak{g}^*)$ and $\beta\in\Lambda^k(\mathfrak{g}^*)$,
\begin{equation}\label{gg'}
\Phi_{\eta}^k(G^{2n-k}_{\zeta}(\alpha), G^{k}_{\zeta}(\beta))=\Phi_{\zeta}^{k}(\beta, \alpha).
\end{equation}

\end{lemma}

\subsection{Almost complex structures and almost Hermitian structures}\label{$J$-decompositions}
Let $J$ be an almost complex structure on $\mathfrak{g}$. Then $J$ induces a natural orientation on $\mathfrak{g}$, which we assume that it is the same as
the fixed one.

Recall that $J$ acts as an involution
 on $\Lambda^2(\mathfrak{g}^*)$, $\alpha(\cdot, \cdot)
\rightarrow \alpha(J\cdot, J\cdot)$. And in Definition \ref{invariant-antiinvariant} we introduce
$\mathcal Z_J^{\pm} = \mathcal Z^2 \cap \Lambda_J^{\pm}(\mathfrak{g}^*)$, and $H_J^{\pm}(\mathfrak{g})$ as the quotient
\begin{equation} \label{pm-quot-desc} {\mathcal Z^{\pm}\over \mathcal B\cap \Lambda^{\pm}_J(\mathfrak{g}^*)}.
\end{equation}
Dually, $J$ also acts on $\Lambda^2(\mathfrak{g})$ as an involution: $J(u\wedge v)=Ju\wedge Jv$.
Accordingly, we also have the $J-$decomposition:
$
\Lambda^2(\mathfrak{g})=\Lambda_J^+(\mathfrak{g})\oplus \Lambda_J^-(\mathfrak{g})\,,
$
which gives rise to
the subgroups $H^J_{\pm} (\mathfrak{g})$ as
 the quotient
\begin{equation} \label{pm-quot-desc-homology} { Z^{\pm}\over B\cap \Lambda^{\pm}_J(\mathfrak{g})}.
\end{equation}

Let $\pi_{\pm}:\Lambda^2(\mathfrak{g})\to \Lambda_J^{\pm}(\mathfrak{g})$ be the projection, and
$$
\Psi_\pm:\Lambda^\pm_J(\mathfrak{g})\times \Lambda^\pm_J(\mathfrak{g}^*)\to \R
$$
be the restriction of $\Psi^2$ to $\Lambda^\pm_J(\mathfrak{g})\times \Lambda^\pm_J(\mathfrak{g}^*)$.

\begin{lemma}\label{closed-exact} We have \smallskip

i). $(\Lambda^\pm_J(\mathfrak{g}), \Lambda^\mp_J(\mathfrak{g}^*))$ is a $\Psi^2-$complementary pair.\smallskip

ii). $\Psi_{\pm}$ is non-degenerate.

iii). $\bar \Psi^2(H_J^{\pm}(\mathfrak{g}), H^J_{\mp}(\mathfrak{g}))=0$.

iv). $(\mathcal Z^\pm, \pi_{\pm}B^2)$ and $(\mathcal B\cap \Lambda^\pm_J(\mathfrak{g}^*), \pi_{\pm}Z^2)$ are $\Psi_\pm-$complementary pairs.

v). There is an induced isomorphism $$\sigma:H_J^\pm(\mathfrak{g})={\mathcal Z^\pm_J
\over \mathcal B\cap \Lambda^\pm_J(\mathfrak{g^*})}\longrightarrow \hom ({\pi_{\pm}Z^2\over \pi_{\pm}B^2}, \R).$$
\end{lemma}

\begin{proof} Since the $J-$actions on $\Lambda^2(\mathfrak{g})$ and $\Lambda^2(\mathfrak{g}^*)$
are involutions and satisfy \eqref{involution-invariance} with respect to the pairing $\Psi^2$,
(i) and (ii) follow from the second and the third bullets of Lemma \ref{involution}.

iii) follows from i) and the description of $H^{\pm}_J(\mathfrak{g})$ in \eqref{pm-quot-desc}.

iv) follows easily from
$(\alpha, \pi_\pm u)=(\alpha, u)$ for $\alpha\in \Lambda^\pm_J(\mathfrak{g}^*)$ and $u\in \Lambda^2(\mathfrak{g})$, and Lemma \ref{closed-exact'}.

v) follows from iv) and the 4th bullet of Lemma \ref{complement}.
\end{proof}

We note that, under $\bar \Psi^2$, the pairings between $H_J^{\pm}(\mathfrak{g})$ and $H^J_{\pm}(\mathfrak{g})$ may not be non-degenerate.
It will be the case when $\mathfrak{g}$ is 4-dimensional and unimodular (Remark \ref{homology decomposition}).

Let now $h$ be an inner product on an oriented $2n$-dimensional Lie algebra $\mathfrak{g}$.
Let $*_h:\Lambda^k(\mathfrak{g}^*)\to \Lambda^{2n-k}(\mathfrak{g}^*)$ be the Hodge-operator associated with $h$, namely,
given $\beta\in\Lambda^k(\mathfrak{g}^*)$,
for every $\alpha\in\Lambda^k(\mathfrak{g}^*)$ define $*_h \beta$ by the following relation
\begin{equation}\label{star}
\alpha\wedge *_h \beta =\langle \alpha,\beta \rangle\hbox{Vol}_h\,,
\end{equation}
where $\langle,\rangle$ is the inner product on forms induced by $h$ and
$\hbox{\rm Vol}_h$ is the
volume form of $h$ with respect to the given orientation.
Let $\eta=\hbox{\rm Vol}_h\in\Lambda^{2n}(\mathfrak{g}^*)$ and choose $\zeta\in\Lambda^{2n}(\mathfrak{g})$ with $\Psi^{2n}(\zeta, \eta)=1$. Then
\begin{equation}\label{phi-star}
\Phi^k_{\zeta}(\alpha, *_h\beta)=\langle \alpha,\beta \rangle.
\end{equation}

For every $\alpha\in\Lambda^p(\mathfrak{g}^*)$ set, as usual,
$$
\delta \alpha= (-1)^{n(p+1)+1} *_h\circ d\circ *_h\,.
$$
Then, set $\Delta = d\delta +\delta d$ and denote by
$$
\mathcal{H}^p({\mathfrak g})=\{\alpha\in\Lambda^p(\mathfrak{g}^*)\,\,\,
\vert\,\,\, \Delta\,\alpha =0\}\,,
$$
the space of $h$-harmonic $p$-forms.

Let $J$ be an almost
complex structure on ${\mathfrak g}$. The inner product $h$ is called $J$-{\em Hermitian} if
$J$ is an orthogonal transformation with respect to $h$.
In this case, the pairing $\varphi$ defined by
$\varphi (\cdot ,\cdot )=h(J\cdot,\cdot)$ lies in
 $\Lambda_J^+(\mathfrak{g}^*)$, and it is called the fundamental form
of $(J,h)$.
The pair $(J, h)$, or equivalently, the triple $(J, h, \varphi)$ is called an {\em almost Hermitian structure}.

With this understood, $J$ is {\em almost K\"ahler} if there is a $J$-Hermitian $h$ such that
the associated fundamental form is closed.
\subsection{Complex structures}
Let $\mathfrak{g}$ be a $2n$-dimensional Lie algebra. An almost complex structure $J$ on $\mathfrak{g}$ is
said to be {\em integrable} if $N_J=0$, where
$$
N_J(u, v)=[u, v]-J[Ju, v]-J[u, Jv]-[Ju, Jv],
$$
for all $u,\,v\in\mathfrak{g}$. Denoting by $\mathfrak{g}_\C=\mathfrak{g}\otimes_\R \C$ the complexification of $\mathfrak{g}$,
then $\mathfrak{g}_\C=\mathfrak{g}^{1,0}\oplus \mathfrak{g}^{0,1}$, where $\mathfrak{g}^{1,0}$, $\mathfrak{g}^{0,1}$, respectively,
are the $+{\bf i}$, $-{\bf i}$ eigenspaces of $J$ acting on $\mathfrak{g}_\C$. Alternatively,
$$
\mathfrak{g}^{1,0}=\{u-{\bf i}Ju\,\,\,\vert\,\,\, u\in \mathfrak{g}\},\quad
\mathfrak{g}^{0,1}=\{u+{\bf i}Ju\,\,\,\vert\,\,\, u\in \mathfrak{g}\}.
$$
Similarly, $\mathfrak{g}_\C^*$ decomposes as $\mathfrak{g}_\C^*=\mathfrak{g}_{1,0}\oplus\mathfrak{g}_{0,1}$, where
$$
\mathfrak{g}_{1,0}=\{\alpha\in\mathfrak{g}_\C^*\,\,\,\vert\,\,\, \alpha(z)=0\,,\,\,\forall z\in\mathfrak{g}^{0,1}\},\quad
\mathfrak{g}_{0,1}=\{\alpha\in\mathfrak{g}_\C^*\,\,\,\vert\,\,\, \alpha(z)=0\,,\,\,\forall z\in\mathfrak{g}^{1,0}\}
$$
Accordingly, the space of complex $r$-forms $\Lambda^r(\mathfrak{g}_\C^*)$ on $\mathfrak{g}^\C$ decomposes as
$$
\Lambda^r(\mathfrak{g}_\C^*) = \bigoplus_{p+q=r}\Lambda^{p,q}_J(\mathfrak{g}_\C^*),
$$
where $\Lambda^{p,q}_J(\mathfrak{g}_\C^*)=\Lambda^p(\mathfrak{g}_{1,0})\wedge \Lambda^q(\mathfrak{g}_{0,1})$.\newline
A straightforward computation shows that $N_J=0$ if and only if $[\mathfrak{g}^{0,1},\mathfrak{g}^{0,1}]\subset \mathfrak{g}^{0,1}$. \newline
It is immediate to prove the following
\begin{prop} \label{integrability} $J$ is integrable if there exists a non-zero closed $(n, 0)$-form on $\mathfrak{g}$.
\end{prop}
It can be also proved that the integrability condition is equivalent to
$$
d\mathfrak{g}_{1,0}\subset \mathfrak{g}_{1,1}\oplus \mathfrak{g}_{2,0},
$$
where the differential $d$ is extended by $\C$-linearity. Let $J$ be an almost complex structure on $\mathfrak{g}$. Set
$\overline{\partial}\vert_{\Lambda^{p,q}_J(\mathfrak{g}_\C^*)}=\pi^{(p,q+1)}\circ d$. Then
$$
\overline{\partial}:\Lambda^{p,q}_J(\mathfrak{g}_\C^*)\to\Lambda^{p,q+1}_J(\mathfrak{g}_\C^*),
$$
and $\overline{\partial}$ can be extended by $\C$-linearity to $\Lambda^r(\mathfrak{g}_\C^*)$. Then, $J$ is integrable if and only if
$\overline{\partial}^2=0$. This allows to define a Dolbeault complex for $\Lambda^{p,q}_J(\mathfrak{g}_\C^*)$. We denote by
$H^{*,*}_{\overline{\partial}}(\mathfrak{g})$ the cohomology of this complex. We refer to it as the {\em Dolbeault cohomology} of
$\mathfrak{g}$.
\subsection{Tamed and almost K\"ahler structures in terms of cone of positive vectors} We characterize tamed and almost K\"ahler
almost complex structures on a Lie algebra in terms of the convex cone of positive vectors
(compare with \cite[Theorem 3.2]{sullivan}, \cite[Theorem 14]{HL} and \cite{LZ}).

Clearly, any simple $2-$vector of the form $v\wedge Jv$ is invariant:
$$
J(v\wedge Jv)=Jv\wedge JJv=-Jv\wedge v=v\wedge Jv.
$$
\begin{definition}\label{positive}
A simple $2-$vector is called {\em positive} if it is of the form $v\wedge Jv$.
A vector in $\Lambda^+_J(\mathfrak{g})$ is said to be {\em positive} if it can be written as a convex combination of $v\wedge Jv$.
\end{definition}

\begin{definition}
We denote by $PC_J$ the convex cone of positive 2-vectors, and we set
$$
HPC_J=\{[u]\,\,\,|\,\,\,u\in PC_J\cap Z\},
$$
which is a convex cone of $H^J_+(\mathfrak{g})$.
\end{definition}
We will need
the following finite dimensional Hahn-Banach lemma, which will be applied to the bilinear pairings $\Psi^2$ and $\Psi^+$ introduced in \ref{bilinearpairings}.
\begin{lemma}\label{Hahn-Banach}
Let $V$ be a finite dimensional vector space and $W$ a linear subspace. If $K$ is a bounded closed convex subset
which is disjoint from $W$. Then there is a linear functional $L$ vanishing on $W$ and strictly positive on $K$.
\end{lemma}
Let $h$ be a $J$-Hermitian metric on $\mathfrak{g}$ and let $\varphi$ be the fundamental $2$-form of $h$. Set
$$
K=\{w\in PC_J\,\,\,\vert\,\,\, T_w(\varphi)=1\}.
$$
Then $K\subset\, PC_J$ is bounded, closed and convex.

\begin{prop}\label{no-invariant-positive-boundary}
Let $J$ be an almost complex
structure on $\mathfrak{g}$.

\begin{itemize}
\item $J$ is tamed if and only if $B^2\cap PC_J=\{0\}$.
There is a bijection between $J-$tamed forms and
functionals on $\Lambda^2(\mathfrak{g})$ vanishing on
${B}^2$ and positive on $PC_J\backslash \{0\}$.
\item $J$ is almost K\"ahler if and only if $\pi_+B^2\cap PC_J=\{0\}$.
There is a bijection between $J-$compatible forms and
functionals on $\Lambda_J^+(\mathfrak{g})$ vanishing on
${\pi_{+}B^2}$ and positive on $PC_J\backslash \{0\}$.
\end{itemize}
\end{prop}
\begin{proof} First of all we prove the characterization of tamed $J$.

$(\Longrightarrow)$. Suppose $J$ is tamed by a symplectic form $\omega$. Then $\Psi^2(B^2, \omega)=0$ by Lemma \ref{closed-exact'}, and $\Psi^2( PC_J\backslash \{0\}, \omega)>0$ by Definitions \ref{tc} and \ref{positive}. Thus $B^2\cap PC_J=\{0\}$.

$(\Longleftarrow)$. Conversely, suppose $B^2\cap PC_J=\{0\}$. Then $B^2\cap K =\emptyset$.
By Lemma \ref{Hahn-Banach}, there is a linear functional $L$ on $\Lambda^2(\mathfrak{g})$, vanishing on  $B^2$ and strictly positive on
$K$. By Lemma \ref{closed-exact'}, $L$ determines a closed
$2$-form, which is positive on $PC_J\backslash \{0\}$.
Such a form is a tamed symplectic form by Definition \ref{tc}.

The argument for characterizing an almost K\"ahler $J$ is similar. The only difference is that, instead of applying Lemma \ref{closed-exact'}
for the pairing $\Psi^2$, we apply Lemma \ref{closed-exact} (iv) for the pairing $\Psi_+$.

$(\Longrightarrow)$. Suppose $J$ is compatible with  a symplectic form $\omega$. Then $\omega\in \mathcal Z^{+}$, and we have $\Psi_+(\pi_+B^2, \omega)=0$ by Lemma \ref{closed-exact} (iv).  Moreover, $\Psi_+( PC_J\backslash \{0\}, \omega)>0$ by Definitions \ref{tc} and \ref{positive}, therefore $\pi_+B^2\cap PC_J=\{0\}$.

$(\Longleftarrow)$. Conversely, suppose $\pi_+B^2\cap PC_J=\{0\}$. Then $\pi_+B^2\cap K =\emptyset$, and
by Lemma \ref{Hahn-Banach}, there is a linear functional $L$ on $\Lambda_J^+(\mathfrak{g})$, vanishing on  $\pi_+B^2$ and strictly positive on
$K$. By Lemma \ref{closed-exact} (iv), $L$ determines a
$2$-form in $\mathcal Z^{+}$, which is positive on $PC_J\backslash \{0\}$.
Such a form is a compatible symplectic form by Definition \ref{tc}.

Finally, it is clear that both the bijection statements  follow from the arguments above.
\end{proof}
\begin{rem}
It has to be noted that complex structures which are tamed by a symplectic form are related to the so called {\em strong K\"ahler metrics with torsion} (shortly {\em SKT metrics}), also called in the literature as {\em pluriclosed} metrics. More precisely, a Hermitian metric on a complex Lie algebra is said to be SKT, if its fundamental form $\omega$ satisfies the condition $\partial\overline{\partial}\omega=0$ (the same definition can be given for an arbitrary complex manifold). The strong K\"ahler metrics with torsion have also applications in type II string theory
and in 2-dimensional supersymmetric $\sigma$-models \cite{GHR,S} and are related to generalized K\"ahler structures (see \cite{G}). \newline
It can be easily proved that if $J$ is a complex structure on a Lie algebra $\mathfrak{g}$ which is tamed by a symplectic form, then $\mathfrak{g}$ admits an SKT metric. For the sake of completeness we record the argument (see e.g. \cite[Proposition 3.4]{AT}). Indeed, let $\omega$ be a symplectic form on $\mathfrak{g}$ taming $J$. Denote
by $\tilde\omega$ the fundamental form of the Hermitian metric
$$
\tilde{g}_J(u,v)=\frac{1}{2}\left(\omega(u,Jv)+
\omega(v,Ju)\right).
$$
Then we easily obtain that
$$
\partial\overline{\partial}\,\tilde\omega =0.
$$
Indeed, by definition, we have
$$
\tilde{\omega} =\frac{1}{2}\left(\omega+ J\omega\right).
$$
Then, we can write
\begin{eqnarray*}
\omega &=& \omega^{2,0} +
\omega^{1,1}+\overline{\omega^{2,0}}\,,\\
\tilde{\omega} &=&\omega^{1,1}\,,
\end{eqnarray*}
where $\overline{\omega^{1,1}}=\omega^{1,1}$. Consequently
$$
d\omega = 0\,\quad\Leftrightarrow \,\quad \left\{
\begin{array}{rl}
\partial \omega^{2,0} & = 0\,,\\[10pt]
\partial\omega^{1,1} + \overline{\partial}\omega^{2,0} &=0\,,
\end{array}
\right.
$$
since $d=\partial +\overline{\partial}$. Therefore,
$$
\partial\overline{\partial}\tilde{\omega}=\partial\overline{\partial}\omega^{1,1}=
-\overline{\partial}\partial\omega^{1,1}=\overline{\partial}^2\omega^{2,0}=0\,.
$$
\end{rem}
\section{Tamed versus almost K\"ahler in dimension $4$}\label{tamed-versus-compatible}
In this section we prove Theorem \ref{tamed-and-compatible-intro}. We begin with reviewing the special feature of $\Lambda^2(\mathfrak{g})$ in dimension $4$.
\subsection{Geometry of $\Lambda^2(\mathfrak{g})$}
Suppose $\mathfrak{g}$ is a $4-$dimensional oriented Lie algebra.
In this case $\Lambda^2(\mathfrak{g})$ and $\Lambda^2(\mathfrak{g}^*)$ have dimension 6. Fix a basis $\{f_1,\ldots ,f_4\}$ of
$\mathfrak{g}$ and denote by $\{f^1,\ldots ,f^4\}$ the dual basis of $\{f_1,\ldots ,f_4\}$. Then
$\{f_i\wedge f_j\}$ and $\{f^i\wedge f^j\}$ are bases of $\Lambda^2(\mathfrak{g})$ and $\Lambda^2(\mathfrak{g}^*)$.

 Let $\zeta$ be a nonzero $4-$vector on
$\mathfrak{g}$ inducing the given orientation, and $\eta$ a volume form with $\Psi^4(\zeta, \eta)=1$.
For ease of notations, denote the pairings $\Psi^2, \Phi_{\zeta}^2, \Phi_{\eta}^2$ by
$\Psi, \Phi_{\zeta}, \Phi_{\eta}$ respectively. Explicitly,
$$\Psi :\Lambda^2(\mathfrak{g})\times\Lambda^{2}(\mathfrak{g}^*)\to\R, \quad \Phi_{\zeta}:\Lambda^2(\mathfrak{g}^*)\times\Lambda^{2}(\mathfrak{g}^*)\to\R, \quad
\Phi_{\eta}:\Lambda^2(\mathfrak{g})\times\Lambda^{2}(\mathfrak{g})\to\R.
$$
Similarly denote $G^2_\eta:\Lambda^2(\mathfrak{g})\to\Lambda^{2}(\mathfrak{g}^*)$
by $G_{\eta}$, and $G^2_\zeta:\Lambda^2(\mathfrak{g}^*)\to\Lambda^{2}(\mathfrak{g})$ by $G_{\zeta}$.

Let $S$ denote the set of simple $2-$vectors on $\mathfrak{g}$.
The following easy observation will be useful.
\begin{lemma} \label{simple=0} Suppose $\mathfrak{g}$ has dimension $4$.
Then $u\in \Lambda^2(\mathfrak{g})$ is simple, i.e. $u\in S$ if and only if $\Phi_{\eta}(u,u)=0$.
\end{lemma}
There is a similar characterization of simple $2$-forms.
Thus $(\Lambda^2(\mathfrak{g}), \Phi_{\eta})$ and $(\Lambda^2(\mathfrak{g}^*), \Phi_{\zeta})$ have signature $(3,3)$.

Let $h$ be an inner product on $\mathfrak{g}$. In dimension $4$,
the Hodge operator $*_h$ on $\Lambda^2({\mathfrak g}^*)$ is another involution on $\Lambda^2({\mathfrak g}^*)$ and induces the well known decomposition
\begin{equation}\label{self-antiself-invariant}
\Lambda^2({\mathfrak g}^*)=\Lambda^+_h({\mathfrak g}^*)\oplus \Lambda^-_h({\mathfrak g}^*),
\end{equation}
where $\Lambda^+_h({\mathfrak g}^*)$, $\Lambda^+_h({\mathfrak g}^*)$ are the vector spaces of
self-dual, anti-self-dual $2$-forms on $\mathfrak{g}$ respectively.
Since $\langle \alpha,\beta\rangle= \langle \beta, \alpha\rangle$ it follows from
the formula \eqref{phi-star} that
\begin{equation}
\Phi_{\zeta}(\alpha, \beta)=\Phi_{\zeta}(*_h\alpha, *_h\beta).
\end{equation}
Thus it follows from Lemma \ref{involution} that
$\Lambda^+_h({\mathfrak g}^*)$ and $ \Lambda^-_h({\mathfrak g}^*)$ are a $\Phi_{\zeta}-$complementary pair, and
$\Phi_{\zeta}$ is non-degenerate on $\Lambda^{\pm}_h({\mathfrak g}^*)$.
In fact, $\Phi_{\zeta}$ is $\pm-$definite on $\Lambda^{\pm}_h({\mathfrak g}^*)$: for $\alpha\in \Lambda^{\pm}_h({\mathfrak g}^*)$,
\begin{equation}\label{star-star}
\Phi_{\zeta}(\alpha, \alpha)=\Phi_{\zeta}(\alpha, \pm*_h\alpha)=\langle \alpha,\pm*_h*_h\alpha\rangle=\pm |\alpha|^2.
\end{equation}
Since ($\Lambda^2({\mathfrak g}^*),\Phi_\zeta)$ has signature $(3, 3)$, $\Lambda^{\pm}_h({\mathfrak g}^*)$ has dimension $3$.

Let $J$ be an almost complex structure on $\mathfrak{g}$.
Then it follows from Lemma \ref{involution} that
$\Lambda_J^+(\mathfrak{g}^*)$ and $\Lambda_J^-(\mathfrak{g}^*)$ are a $\Phi_{\zeta}-$complementary pair, and
$\Phi_{\zeta}$ is non-degenerate on $\Lambda_J^{\pm}(\mathfrak{g}^*)$.
The same is true for $\Lambda_J^{\pm}(\mathfrak{g})$ with respect to $\Phi_{\eta}$.

Suppose further $h$ is $J$-Hermitian and $\varphi\in \Lambda^2(\mathfrak{g}^*)$ is the fundamental form
of $(J,h)$.
In this case, $\eta={1\over 2}\varphi\wedge \varphi$. One can easily verify the following
relations among the subspaces $\Lambda^{\pm}_{J}({\mathfrak g}^*), \Lambda^{\pm}_{h}({\mathfrak g}^*), \Span\{\varphi\}$ of $\Lambda^2({\mathfrak g}^*)$:
\begin{equation}\label{decomposition-invariant-1}
\Lambda^+_{J}({\mathfrak g}^*)=\R(\varphi)\oplus\Lambda^-_{h}({\mathfrak g}^*),\quad
\Lambda^+_{h}({\mathfrak g}^*)=\R(\varphi)\oplus\Lambda^-_{J}({\mathfrak g}^*),
\end{equation}
and
\begin{equation}\label{decomposition-invariant-2}
\Lambda^+_{J}({\mathfrak g}^*)\cap \Lambda^+_{h}({\mathfrak g}^*)=\R(\varphi),\quad
\Lambda^-_{J}({\mathfrak g}^*)\cap \Lambda^-_{h}({\mathfrak g}^*)=\{0\}.
\end{equation}
The following lemma is straightforward.
\begin{lemma}\label{invariant-pairing} Suppose $\mathfrak{g}$ has dimension $4$ and $J$ is an almost complex structure on $\mathfrak{g}$.
Consider the pairings $\Phi_{\eta}, \Phi_\zeta$ and the isomorphisms $G_{\eta}, G_{\zeta}$. We have

(i) $G_{\zeta}(\Lambda_J^{\pm}(\mathfrak{g}^*))=\Lambda_J^{\pm}(\mathfrak{g})$, and $\Phi_{\eta}(G_{\zeta}(\alpha), G_{\zeta}(\beta))=\Phi_\zeta(\beta, \alpha)$.

(ii) $\Lambda_J^+(\mathfrak{g})$ and $\Lambda_J^+(\mathfrak{g}^*)$ have signature $(1, 3)$.

(iii) $\Lambda_J^-(\mathfrak{g})$ and $\Lambda_J^-(\mathfrak{g}^*)$ are positive definite.

(iv) If $v\in PC_J$,
then $\Phi_{\eta}(v,v)\geq 0$.
\end{lemma}
\begin{proof}

(i) The first formula follows from the above mentioned fact that $(\Lambda_J^+(\mathfrak{g}^*), \Lambda_J^-(\mathfrak{g}^*))$ is a
$\Phi_{\zeta}-$complementary pair,
 and the first part of Lemma \ref{closed-exact}. The second is a special case of the fourth part of Lemma \ref{pairing}.
\newline
(ii) and (iii) follow from \eqref{decomposition-invariant-1} and (i). \newline
(iv) follows from $\Phi_{\eta}(f_i\wedge Jf_i, f_j\wedge Jf_j)\geq 0$ for any $i, j$, which is straightforward.
\end{proof}
\subsection{Tamed versus almost K\"ahler}
We will apply Proposition \ref{no-invariant-positive-boundary} to prove Theorem \ref{tamed-and-compatible-intro}.
The following observations on the set $S$ of simple vectors are crucial for this purpose.

\begin{lemma} \label{invariant-vectors} For an almost complex structure $J$ on $\mathfrak{g}$ of arbitrary dimension, we have
\begin{equation}
S\cap\Lambda^+_J(\mathfrak{g})\subset PC_J\sqcup \,-PC_J.
\end{equation}

\end{lemma}
\begin{proof}
Observe that if
a simple $2$-vector is in $\Lambda^+_J(\mathfrak{g})$ then it is of the form $\pm(v\wedge Jv)$, and in fact, $\Lambda^+_J(\mathfrak{g})$ is generated by simple vectors of the form $v\wedge Jv$ .
\end{proof}

The second one is
only valid in dimension $4$.
\begin{lemma}\label{invariant-part}
Suppose $\mathfrak{g}$ has dimension $4$. If $v\ne 0\in S$ and $v\notin \Lambda_J^+(\mathfrak{g})$, then
$\pi_+v\notin PC_J$.
\end{lemma}

\begin{proof}


Let $v\in S$ and $v=\pi_+v + \pi_-v$. By assumption, $\pi_-v\ne 0$, and hence $\Phi_{\eta}(\pi_-v, \pi_- v)>0$ by Lemma \ref{invariant-pairing} (iii).
Since $$0=\Phi_{\eta}(v,v)=\Phi_{\eta}(\pi_+v, \pi_+ v) +\Phi_{\eta}(\pi_-v, \pi_- v) ,$$
$\Phi_{\eta}(\pi_+v, \pi_+ v)<0$.
The conclusion follows from Lemma \ref{invariant-pairing} (iv).
\end{proof}
We are ready to show the following theorem, from which Theorem \ref{tamed-and-compatible-intro} follows.
\begin{theorem}\label{compatible}
Let $\mathfrak{g}$ be a $4-$dimensional Lie algebra with $B\wedge B=0$. Let $J$ be an
almost complex structure on $\mathfrak{g}$. The following are equivalent.
\begin{enumerate}
\item[0)]
$J$ is tamed. \vskip.1truecm\noindent
\item[1)]
$B\cap PC_J=\{0\}$. \vskip.1truecm\noindent
\item[2)] $B\cap\Lambda^+_J(\mathfrak{g})=\{ 0 \}$.\label{no-invariant-boundary-2}\vskip.1truecm\noindent
\item[3)] $\pi_+B\cap PC_J=\{ 0 \}$.\label{no-invariant-boundary-3}\vskip.1truecm\noindent
\item[4)] $J$ is almost K\"ahler. \label{tamed-and-compatible}
\end{enumerate}
\end{theorem}
\begin{proof}
 $0)=1)$ is the first bullet of Proposition \ref{no-invariant-positive-boundary}.
\vskip.2truecm
$1)=2)$
 Since $PC_J\subset \Lambda^+_J(\mathfrak{g})$ by definition, we need to show only $(\Longrightarrow)$, namely, if there is a nonzero vector in
 $B\cap\Lambda^+_J(\mathfrak{g})$, then there is a nonzero vector in $B\cap PC_J$.
 \newline
 Let $0\neq v\in B\cap\Lambda^+_J(\mathfrak{g})$. Then
 by our assumption $v\wedge v=0$. Now it follows from Lemma \ref{simple=0} that
 $v\in S$.
Therefore $v\in S\cap\Lambda^+_J(\mathfrak{g})$. Hence, by Lemma \ref{invariant-vectors},
$v$ or $-v$ $\in PC_J$. Thus $B\cap PC_J$ contains either $v$ or $-v$.
\vskip.2truecm

$2)\Longrightarrow 3)$ Suppose $0\neq v\in B$, then $v\notin \Lambda^+_J(\mathfrak{g})$ by assumption.
Since $v\in S$, by Lemma \ref{invariant-part}, $\pi_+v\notin PC_J$.
Therefore $\pi_+B\cap PC_J=\{ 0 \}$.
\vskip.2truecm
$3)=4)$ is the second bullet of Proposition \ref {no-invariant-positive-boundary}.
\vskip.2truecm
$4) \Longrightarrow 0)$ An almost K\"ahler $J$ is tamed by definition.
\end{proof}
It is easy to see that $B\wedge B=0$ for any unimodular $4-$dimensional Lie algebra (see part 3 of Lemma \ref{exact-boundary}).

In \cite{EF} it is proved that a $4$-dimensional Lie algebra $\mathfrak{g}$ endowed
with a complex structure $J$ admits a taming symplectic structure if and only if $(\mathfrak{g},J)$ has a K\"ahler metric.
It is interesting to study whether Theorem \ref{compatible} offers an alternative proof of this result. \vskip.2truecm\noindent
\subsection{Positive $2-$planes and almost complex structures}
We have shown that on a $4-$dimensional unimodular Lie algebra, all tamed $J$ are almost K\"ahler.
In the next section we will determine which $J$ is tamed.
For this purpose we first review in this subsection Donaldson's description of $4-$dimensional almost complex structures
in terms of positive $2-$planes in $\Lambda^2(\g^*)$ (\cite{D}). Then we calculate explicitly the positive $2-$planes $J\to \Lambda_J^-(\g^*)$
given a basis of $\g$.

Notice that both $\hbox{SL}(4,\R)$ and $\hbox{SO}(3,3)$ have dimension $15$. In fact, there is a natural
representation $\pi :\hbox{SL}(4,\R)\to\hbox{Aut}(\Lambda^2(\mathfrak{g}), \Phi_{\eta} )$ of $\hbox{SL}(4,\R)$ on $\Lambda^2(\mathfrak{g})$:
given any $A\in\hbox{SL}\,(4,\R)$, set
$\pi(A)(\sum p^{ik}f_i\wedge f_k)=\sum p^{ik}A f_i\wedge A f_k.$

It is immediate to check that
$\pi :\hbox{SL}\,(4,\R)\to\hbox{Aut}\,(\Lambda^2(\mathfrak{g}))$ is a group homomorphism, and $\pi(A)$ preserves $\Phi_{\eta}$ since $\det(A)=1$.
Since $\hbox{SL}\,(4,\R)$ is connected, $\pi(\hbox{\rm SL}\,(4,\R))\subset \hbox{Aut}^0(\Lambda^2(\mathfrak{g}), \Phi_{\eta})\simeq \hbox{SO}^0(3,3)\,,$
where the superscript $0$ means the connected component of the identity.
Furthermore, we have that $\Ker(\pi)=\pm I$.
Indeed, let $A\in\Ker(\pi)$; then $\pi(A)(f_i\wedge f_j)=(Af_i\wedge Af_j)=f_i\wedge f_j\,,i,j=1,\ldots ,4$. Hence,
by part 4 of Lemma \ref{pairing}, for any fixed
$i$, the vector $Af_i\in\Span\{f_i,f_j\}\,,j=1,\ldots ,4$. Consequently, $A=\lambda I$. Since $A\in \hbox{SL}(4,\R)$, it follows that $\lambda=\pm 1$, i.e.,
$A=\pm I$.\bigskip

Dualizing $\pi: \hbox{\rm SL}\,(4,\R)\to \hbox{\rm SO}^0(\Lambda^2(\mathfrak{g}), \Phi_{\eta})$ we obtain the double
covering $\pi^*: \hbox{\rm SL}\,(4,\R)\to \hbox{\rm SO}^0(\Lambda^2(\mathfrak{g}^*), \Phi_{\zeta})$.
One important consequence is that $\hbox{\rm SL}(4,\R)$ acts transitively on the set of $\Phi_{\zeta}-$positive definite
$k-$planes for $k=1,2,3$. As observed by Donaldson in \cite{D}, such $k-$planes correspond to various
geometric structures. In particular,

\begin{prop}\label{2-plane} (\cite{D})
For a $4-$dimensional Lie algebra $\g$, the map $J\to \Lambda_J^-(\g^*)$ is a one-to-one correspondence between the space of
 almost complex structures and the space of oriented positive definite two plane in $\Lambda^2(\mathfrak{g}^*)$.
\end{prop}
This point of view is useful to detect integrable $J$.
\begin{lemma} \label{integrability-dim4} Suppose $\dim \mathfrak{g}=4$. Then $J$ is integrable if $\Lambda_J^-(\mathfrak{g}^*)\subset \mathcal Z$.
\end{lemma}
\begin{proof}
Suppose $\alpha\in \Lambda_J^-$. Let us consider the additional action of $s_J$ of $J$ on $\Lambda^2(\mathfrak{g}^*)$, defined as
$$
s_J\alpha(u, v)=\alpha(Ju, v).
$$
This action preserves $\Lambda_J^-$. Then, by the assumption, both $\alpha$ and $s_J\alpha$ are closed. Therefore $\alpha+{\bf i} (s_J\alpha)$ is a
closed $(2,0)-$form and the statement follows from Proposition \ref{integrability}.
\end{proof}
Fix a basis $\{f_i\}$ of $\mathfrak{g}$. We present an explicit covering of the 8-dimensional space of almost complex structures,
and an explicit calculation of $\Lambda_J^-(\mathfrak{g}^*)$
with respect to this covering.

Suppose
$$
Jf_k=\sum_{h=1}^4a_{hk}f_h\,,\quad a_{hk}\in\R.
$$ Assume that $\{f_1,Jf_1,f_3, Jf_3\}$ is still
a basis of $\mathfrak{g}$. Then
\begin{equation}\label{anti-invariant-covectors}
\Lambda_J^-(\mathfrak{g}^*)=\Span\left\{f^{24}-\sum_{h<k}p_{hk}f^{hk}\,,\,\sum_{h}\left(a_{2h}f^{4h}- a_{4h}f^{2h}\right)\right\}.
\end{equation}
where
$$
p_{hk}=\det
\begin{pmatrix}
a_{2h} & a_{2k}\\
a_{4h} & a_{4k}
\end{pmatrix}
\qquad h,k=1,\ldots ,4\,,h<k,
$$
are the Pl\"ucker coordinates of the vectors $(a_{2i})$ and $(a_{4j})$.
The previous computation is general. Indeed, the condition that
$\{f_1,Jf_1, f_3, Jf_3\}$ is a basis is equivalent to say that the $2$-plane $L_{13}=\Span\{f_1,f_3\}$ is not $J$-invariant. Set, for any pair of
indices $(i,j)$, with $i<j$, $L_{ij}=\Span\{f_i,f_j\}$ and
$$
\mathcal{U}_{ij}=\{J\,\,\,\vert\,\,\, L_{ij}\,\,\hbox{\rm is not }\,J-\hbox{\rm invariant}\}.
$$
Since $L_{ij}$ cannot be $J-$invariant for all pairs, $\{\mathcal{U}_{ij}\}$ is a covering of the space of the almost complex structures on $\mathfrak{g}$.
It can be proved that
$$
\mathcal{U}_{ij}\simeq \{(a_i,a_j)\in\R^4\times\R^4\,\vert\,p_{\overline{i}\,\overline{j}}=\det
\begin{pmatrix}
a_{\overline{i}i} & a_{\overline{i}j}\\
a_{\overline{j}i} & a_{\overline{j}j}
\end{pmatrix}\neq 0\,\,\hbox{\rm where},\overline{i},\,\overline{j}\in\{1,\ldots ,4\}\setminus\{i,j\},\,
\overline{i}<\overline{j}\}.
$$
\section{Unimodular 4-dimensional Lie algebras}\label{unimodular4-dimensional}
In this section $\g$ is a unimodular Lie algebra.
\subsection{Unimodularity}
We start by recalling the following
\begin{definition}
A $2n$-dimensional Lie algebra $\mathfrak{g}$ is said to be {\em unimodular} if
$\hbox{\rm tr}(\hbox{\rm ad}_\xi)=0$ for every
$\xi\in\mathfrak{g}$.
\end{definition}

It turns out that $\mathfrak{g}$ is unimodular if and only if
\begin{equation}\label{unimodular}
d(\Lambda^{2n-1}\left(\mathfrak{g}^*)\right)=0
\end{equation}
(see \cite[p. 81 (6.3)]{K}). Equivalently, $\mathfrak{g}$ is unimodular if and only if $b_{2n}(\mathfrak{g})=1$.
\begin{lemma} \label{exact-boundary} Let $\mathfrak{g}$ be a $2n$-dimensional unimodular Lie algebra. Then\vskip.1truecm
1. $G^k_\eta(Z_k)=\mathcal{Z}^{2n-k}$.\vskip.1truecm

2. $G^k_\eta(B_k)=\mathcal{B}^{2n-k}$.\vskip.1truecm

3. $(\mathcal{Z}^k, \mathcal{B}^{2n-k})$ and $({Z}^k, {B}^{2n-k})$ are $\Phi^k_\zeta$-complementary pairs.

4. There are induced non-degenerate pairings
$$
\bar \Phi^k_\zeta:H^k(\mathfrak{g})\times H^{2n-k}(\mathfrak{g})\to \R, \quad \bar \Phi^k_\eta:H_k(\mathfrak{g})\times H_{2n-k}(\mathfrak{g})\to \R.
$$
\vskip.1truecm

5. Furthermore, let $h$ be an inner product on $\mathfrak{g}$ with the volume form $\hbox{\rm Vol}_h$. Then, for every
$\alpha\in\Lambda^{p-1}(\mathfrak{g}^*)$,
$\beta\in\Lambda^p(\mathfrak{g}^*)$ we have
$$
\langle d\alpha,\beta\rangle = \langle \alpha,\delta \beta\rangle,
$$
and the following Hodge decomposition holds
\begin{equation}
\label{Hodge-Invariant-p-dimensional}
\Lambda^p({\mathfrak g}^*)=\mathcal{H}^p({\mathfrak g})\oplus d(\Lambda^{p-1}({\mathfrak g}^*))\oplus
\delta(\Lambda^{p+1}({\mathfrak g}^*))\,.
\end{equation}
\end{lemma}
\begin{proof}

$1)$ Let $u\in Z_k$ be a $k$-cycle and let $\alpha=G^k_\eta(u)\in\Lambda^{2n-k}(\mathfrak{g}^*)$. Then, for every
$\gamma\in\Lambda^{k-1}(\mathfrak{g}^*)$, $d(\gamma\wedge \alpha)=0$
since $\mathfrak{g}$ is unimodular. By applying the third part of lemma \ref{pairing}, we have
$$
T_{\zeta}(\gamma\wedge d\alpha)=(-1)^kT_{\zeta}\left(d\gamma\wedge\alpha\right)=(-1)^k T_{G^{2n-k}_\zeta(\alpha)}(d\gamma)=
(-1)^kT_u(d\gamma)=-dT_u(\gamma)=0,
$$
and, consequently, $d\alpha=0$.\vskip.3truecm\noindent
$2)$ Let $u\in B^k$ be a $k$-boundary, i.e., $T_u=dT_w$, for a $(k+1)-$vector $w$. \newline
Set
$\beta=G^{k+1}_\eta(w)\in\Lambda^{2n-k-1}(\mathfrak{g}^*)$. Then, by applying the third part of lemma
\ref{pairing} twice, for every $\gamma\in\Lambda^{k}(\mathfrak{g}^*)$,
we have
$$
\begin{array}{lll}
T_{\zeta}(\gamma\wedge\alpha)&=&T_{G^{2n-k}_\zeta(\alpha)}(\gamma)=T_u(\gamma)=dT_w(\gamma)=(-1)^{k+1}T_w(d\gamma)=
(-1)^{k+1}T_{G^{2n-k-1}_\zeta(\beta)}(d\gamma)\\[5pt]
&=& (-1)^{k+2}T_{\zeta}(d\gamma\wedge\beta)=-T_{\zeta}(\gamma\wedge d\beta);
\end{array}
$$
therefore, $\alpha=d(-\beta)$.\vskip.3truecm\noindent
$3)$ By part 4 of Lemma \ref{pairing} and 1) and 2) it suffices to prove $\mathcal{Z}^k$ and $\mathcal{B}^{2n-k}$ are $\Phi^k_\zeta$-complementary pair.
Further, by Lemma \ref{complement} it is enough to prove that
$$
\mathcal{Z}^k=\left(\mathcal{B}^{2n-k}\right)^\bot\,,
$$
where $\bot$ is taken with respect to $\Phi^k_\zeta$.

Let $\alpha\in\mathcal{Z}^k$ and $\beta\in\mathcal{B}^{2n-k}$, $\beta =d\gamma$; then, we have
$$
\Phi_\zeta^k(\alpha,d\gamma)=(\alpha\wedge d\gamma)(\zeta)=(-1)^{k+1}(d\alpha\wedge \gamma)(\zeta)=0,
$$
i.e., $\Phi_\zeta^k(\mathcal{Z}^k,\mathcal{B}^{2n-k})=0$. Therefore, $\mathcal{Z}^k\subset\left(\mathcal{B}^{2n-k}\right)^\bot$.\newline
Conversely, let $\alpha\in\left(\mathcal{B}^{2n-k}\right)^\bot$.
Then, for any $\gamma\in\Lambda^{2n-k-1}(\mathfrak{g}^*)$, we have
$$
0=\Phi^k_\zeta(\alpha,d\gamma)=(\alpha\wedge d\gamma)(\zeta)=(-1)^{k+1}(d\alpha\wedge\gamma)(\zeta)\,;
$$
hence, $(d\alpha\wedge\gamma)(\zeta)=0$, for every $\gamma\in\Lambda^{2n-k-1}(\mathfrak{g}^*)$. \newline
Therefore, $d\alpha=0$ and $\left(\mathcal{B}^{2n-k}\right)^\bot\subset\mathcal{Z}^k$.\vskip.2truecm
\noindent
$4)$ It follows from 3) and the last bullet of Lemma \ref{complement}.
\vskip.3truecm\noindent
$5)$ Consider the linear map $T_{\zeta}:\Lambda^{2n}(\mathfrak{g^*})\to\R$ with $\zeta$ satisfying
$\hbox{\rm Vol}_h(\zeta)=1$. Since $\alpha\wedge *_h\beta$ is a $(2n-1)$-form and $\mathfrak{g}$ is unimodular, $ d(\alpha\wedge *_h\beta)=0$. Thus we have
$$0=T_{\zeta}\left( d(\alpha\wedge *_h\beta)\right)=\langle d\alpha,\beta\rangle- \langle \alpha,\delta \beta\rangle.$$
This implies that $\Delta$ is self-adjoint, and $\Delta\,\alpha =0$ if and only
if $d\alpha =0=\delta\alpha.$
Since the complex $(\Lambda^*({\mathfrak g}),d)$ has finite dimension and
$\Delta$ is self-adjoint, it follows at once that
$$
\Lambda^p({\mathfrak g}^*) =\hbox{\rm Ker}\,\Delta\oplus \Image\Delta,
$$
which implies \eqref{Hodge-Invariant-p-dimensional}.
\end{proof}
\subsection{Cohomological $J$-decompositions}
Let $\mathfrak{g}$ be a $4$-dimensional unimodular Lie algebra endowed with an inner product $h$. Let $\mathcal{H}({\mathfrak g})$
be the space of $h$-harmonic $2$-forms
on ${\mathfrak g}$. Then, by \eqref{Hodge-Invariant-p-dimensional}, we have the Hodge decomposition
\begin{equation}
\label{Hodge-Invariant}
\Lambda^2({\mathfrak g}^*)=\mathcal{H}({\mathfrak g})\oplus d(\Lambda^{1}({\mathfrak g}^*))\oplus
d^*(\Lambda^{3}({\mathfrak g}^*))\,.
\end{equation}
Let $\mathcal{H}^+({\mathfrak g})$, $\mathcal{H}^-({\mathfrak g})$ be the space of
$h$-harmonic self-dual, anti-self-dual forms respectively on ${\mathfrak g}$.
Since $*_h$ commutes with $\Delta$ we obtain that
$$
\mathcal{H}({\mathfrak g})=\mathcal{H}^+({\mathfrak g})\oplus\mathcal{H}^-({\mathfrak g})
$$
Set
$$
H^\pm_h({\mathfrak g})=\{a\in H^2({\mathfrak g})\,\,\,\vert\,\,\, \hbox{\rm there exists}\,
\alpha\in \mathcal{Z}^\pm_h\,\,
\hbox{\rm such that}\, a=[\alpha]\}\,,
$$
where $\mathcal{Z}^\pm_h=\mathcal{Z}^2\cap\Lambda^\pm_h(\mathfrak{g})$.
Then, by \eqref{Hodge-Invariant}, the inclusion map $\mathcal{H}({\mathfrak g})\hookrightarrow \Lambda^2(\mathfrak{g}^*)$ gives rise to an
isomorphism between $\mathcal{H}({\mathfrak g})$ and $H^2(\mathfrak{g})$ such that
$\mathcal{H}^+({\mathfrak g})\cong H^+_h({\mathfrak g})$,
$\mathcal{H}^-({\mathfrak g})\cong H^-_h({\mathfrak g})$ and
$$
H^2(\mathfrak{g})=H^+_h({\mathfrak g})\oplus H^-_h({\mathfrak g})\,.
$$

Let us consider the
cohomology groups $H^+_J(\mathfrak{g})$, $H^-_J(\mathfrak{g})$
associated with an almost complex structure $J$.
\begin{theorem}\label{Nomizu-1-Lie-algebra}
Let ${\mathfrak g}$ be a unimodular $4$-dimensional Lie algebra ${\mathfrak g}$ endowed with an
almost complex structure $J$. Then the following decomposition
holds:
\begin{equation}\label{formula-Nomizu-1}
H^2(\mathfrak{g})=H^+_{J}({\mathfrak g})\oplus H^-_{J}({\mathfrak g}).
\end{equation}
\end{theorem}
\begin{proof}
First of all we show that $H^+_{J}({\mathfrak g})\cap H^-_{J}({\mathfrak g})=\{0\}$. Indeed,
let $0\neq \mathfrak{a}\in H^+_J(\mathfrak{g})\cap
H^-_J(\mathfrak{g})$. Then $\mathfrak{a}=[\alpha]=[\beta]$, where $\alpha\in \mathcal{Z}^+_
J$,
$\beta\in \mathcal{Z}^-_J$ respectively and $\beta=\alpha +d\gamma$. Let $h=\langle, \rangle$ be a $J$-Hermitian inner product on $\mathfrak{g}$,
denote by $\varphi$ the fundamental form of $(J,h)$. Since the inner product is $J$-invariant, it follows immediately that
$\langle \alpha,\beta\rangle =0$. On the other hand, by \eqref{decomposition-invariant-1}, we have that
$\Lambda^-_{J}({\mathfrak g}^*)\subset \Lambda^+_{h}({\mathfrak g}^*)$. Consequently, $\beta$ is self-dual harmonic.
Then,
$$
0<\vert \beta\vert^2 = \langle \beta,\beta\rangle = \langle \beta,\alpha+d\gamma\rangle =
\langle \beta,d\gamma\rangle =0\,.
$$
Hence $\mathfrak{a}=0$.

Now we have to show that
\begin{equation}\label{+}
H^2(\mathfrak{g})=H^+_J({\mathfrak g})+ H^-_{J}({\mathfrak g}).
\end{equation}
Assume on the contrary that \eqref{+} is not true. Then we can find a cohomology class
$0\neq \mathfrak{a}\in H^2(\mathfrak{g})$ such that
$\mathfrak{a}$ is $\bar \Phi_{\zeta}-$orthogonal to
$H^+_{J}({\mathfrak g})+ H^-_{J}({\mathfrak g})$. We may assume
that $\mathfrak{a}\in H^+_{h}({\mathfrak g})$, since
$H^-_{h}({\mathfrak g})\subset H^+_{J}({\mathfrak g})$. Let
$\alpha$ be the $h$-harmonic representative of $\mathfrak{a}$.
Since $\mathfrak{a}\in H^+_{h}({\mathfrak g})$, $\alpha$ is self-dual.
Set $c=\langle \alpha, \varphi\rangle$. Then $c\neq 0$, otherwise
$a\in H^-_{J}({\mathfrak g})$.\medskip

We will use the following decomposition properties, which are the Lie algebraic counterpart of Lemma 2.4 of \cite{DLZ}:
if $\alpha\in\Lambda^+_h({\mathfrak g})$ and $\alpha =\alpha_{\h} + d\lambda + \delta\eta$ is its Hodge decomposition
according to \eqref{Hodge-Invariant}, then
\begin{equation}\label{Hodge-4-dimensional-1}
(d\lambda)^+_h=(\delta\eta)^+_h\,\quad (d\lambda)^-_h=-(\delta\eta)^-_h,
\end{equation}
\begin{equation}\label{Hodge-4-dimensional-2}
\alpha -2(d\lambda)^+_h=\alpha_{\h},
\end{equation}
\begin{equation}\label{Hodge-4-dimensional-3}
\alpha +2(d\lambda)^-_h=\alpha_{\h}+2d\lambda,
\end{equation}
where, according to the decomposition \eqref{self-antiself-invariant}, any
$\beta\in\Lambda^2({\mathfrak g}^*)$ can be written as $\beta=\beta^+_h+\beta^-_h$ with $\beta^{\pm}_h\in\Lambda^{\pm}_h({\mathfrak g}^*)$.

For the sake of completeness, we
give the proof of \eqref{Hodge-4-dimensional-1}, \eqref{Hodge-4-dimensional-2}
(the proof of \eqref{Hodge-4-dimensional-3} being similar).
 Denote $*_h$ by $*$. Since $*\alpha=\alpha$, we get
$$
*\alpha_\h + *d\lambda + *\delta \eta=\alpha_\h + d\lambda + \delta\eta.
$$
By the uniqueness of Hodge decomposition \eqref{Hodge-Invariant-p-dimensional}, we obtain
\begin{equation}\label{star-4-dim}
*d\lambda = \delta\eta\,,\quad *\delta \eta = d\lambda.
\end{equation}
Therefore, by using \eqref{star-4-dim}, we get
$$
(d\lambda)^+_h=\frac{1}{2}(d\lambda + *d\lambda)=\frac{1}{2}(*\delta\eta + \delta\eta)=(\delta\eta)^+_h
$$
and
$$
(d\lambda)^-_h=\frac{1}{2}(d\lambda - *d\lambda)=\frac{1}{2}(*\delta\eta - \delta\eta)=-(\delta\eta)^-_h
$$
i.e., \eqref{Hodge-4-dimensional-1} is proved.\smallskip

By using again \eqref{star-4-dim}, we have
$$
\alpha - 2(d\lambda)^+_h = \alpha - (d\lambda + *d\lambda) =\alpha_\h
$$
and \eqref{Hodge-4-dimensional-2} is proved. \bigskip

Now, applying
\eqref{Hodge-4-dimensional-3} to the self-dual form $c\,\varphi$, we obtain:
$$
c\,\varphi + 2((c\,\varphi)_{{exact}})^-_h=(c\,\varphi)_{\h} +
2(c\,\varphi)_{{exact}}\in \mathcal{Z}^+_{J}\,,
$$
where $(\alpha)_{exact}$ denotes the exact part from of the form $\alpha$ from the Hodge decomposition.
Setting $\mathfrak{b}=[(c\,\varphi)_{\h}]\in H^+_{J}({\mathfrak g})$, we get
$$
0=\overline{\Phi}_\zeta(\mathfrak{a},\mathfrak{b}) =
{\Phi}_\zeta(\alpha,(c\,\varphi)_{\h} + 2(c\,\varphi)_{{exact}}) =
{\Phi}_\zeta(\alpha,c\,\varphi + 2((c\,\varphi)_{{exact}})^-_h) =c^2.
$$
This is absurd, since $c\neq 0$.
\end{proof}

\begin{rem}\label{counterexamples}
If either the condition being $4-$dimensional or unimodular is dropped, then \eqref{+} may not hold. For higher dimensional examples
see \cite{FT}.
\end{rem}

\begin{rem}\label{homology decomposition}
Theorem \ref{Nomizu-1-Lie-algebra} implies that the pairings between
$H_J^{\pm}(\mathfrak{g})$ and $H^J_{\pm}(\mathfrak{g})$ are non-degenerate, and
there is also a homological $J-$decomposition: $H_2(\mathfrak{g})=H^J_{+}({\mathfrak g})\oplus H^J_{-}({\mathfrak g})$. \newline
Indeed,
\begin{itemize}
\item $H^J_{+}(\mathfrak{g})\cap H^J_{-}(\mathfrak{g})=\{0\}$. Let
$\overline{\Psi} : H_2(\mathfrak{g})\times H^2(\mathfrak{g}) \to \R$ be the non-degenerate pairing as in Lemma \ref{closed-exact'}. Let
$[u]\in H^J_{+}(\mathfrak{g})\cap H^J_{-}(\mathfrak{g})$. Then
$$
\overline{\Psi}([u],\cdot)_{\vert_{H^+_J(\mathfrak{g})}}=0\,,\quad \overline{\Psi}([u],\cdot)_{\vert_{H^-_J(\mathfrak{g})}}=0.
$$
Since $H^2(\mathfrak{g})=H^+_J(\mathfrak{g})+H^-_J(\mathfrak{g})$ and $\overline{\Psi}$ is non-degenerate, it follows that $[u]=0$.\vskip.2truecm\noindent
\item $H_2(\mathfrak{g})= H^J_{+}(\mathfrak{g})+ H^J_{-}(\mathfrak{g})$. Let $\bar \Phi_\zeta:H^2(\mathfrak{g})\times H^{2}(\mathfrak{g})\to \R$ be
the non-degenerate pairing as in Lemma \ref{exact-boundary} and let
$\bar\Psi:\hbox{\rm hom}(H^2(\mathfrak{g}),\R)\to H_2(\mathfrak{g})$ be the isomorphism induced by the non-degenerate pair $\bar\Psi$ of Lemma
\ref{closed-exact'}. Then, $\bar\Psi\circ\bar \Phi_\zeta:H^2(\mathfrak{g})\to H_2(\mathfrak{g})$ is an
isomorphism. In fact this map is induced by
$G_\zeta:\Lambda^2(\mathfrak{g}^*)\to \Lambda^2(\mathfrak{g})$. It then follows from the properties of $G_{\zeta}$ in the first part of Lemma \ref{invariant-pairing} and parts 1 and 2 of Lemma \ref{exact-boundary}
that $\bar\Psi\circ\bar \Phi_\zeta$ induces inclusions
$
H^{\pm}_J(\mathfrak{g})\hookrightarrow H^J_\pm(\mathfrak{g}).
$
Since $H^2(\mathfrak{g})$ and $H_2(\mathfrak{g})$ have the same dimension, we have $H_2(\mathfrak{g})= H^J_{+}(\mathfrak{g})\oplus H^J_{-}(\mathfrak{g})$, and
\item $\bar \Psi^{\pm}: H^{\pm}_J(\mathfrak{g})\times H_{\pm}^J(\mathfrak{g})\to \R$ is non-degenerate.
\end{itemize}
\end{rem}
\subsection{Cohomological characterization of tamed $J$}
In this subsection we are going to prove Theorem \ref{tame-char}.

For a linear subspace $V\subset \Lambda^2(\mathfrak{g}^*)$, let $n(V)$ be the maximal dimension of an isotropic subspace with respect to $\Phi_{\zeta}$,
and $b^+(V)$ be the maximal dimension of a positive definite subspace with respect to $\Phi_{\zeta}$.
We use $b^+(\mathfrak{g})$ for $b^+(\mathcal Z)$. We have the following
\begin{lemma} \label{positive-boundary}
Let $\mathfrak{g}$ be a $4-$dimensional unimodular Lie algebra.

(i) $\dim \mathcal Z=b^+(\mathfrak{g})+3$ and $n(\mathcal Z)=3-b^+(\mathfrak{g})$.

(ii) \label{boundary-simple} $\dim \mathcal B=n(\mathcal B)=3-b^+(\mathfrak{g})$.

(iii) $\dim(\Lambda_J^+(\mathfrak{g}^*)\cap \mathcal{B}) \leq 1.$

(iv) \label{+Z} $\dim (\Lambda^+_J(\mathfrak{g}^*)\cap \mathcal{Z})=b^+(\mathfrak{g})+1.$

(v) $\mathcal B\cap \Lambda_J^-(\mathfrak{g}^*)=\{0\}$.

(vi) $b^+(\mathfrak{g})-1\leq h_J^-=\dim (\Lambda^-_J(\mathfrak{g}^*)\cap \mathcal{Z})\leq b^+(\mathfrak{g}).$
\end{lemma}
\begin{proof}
(i)+(ii) It follows from the third part of Lemma \ref{exact-boundary} that $(\mathcal Z, \mathcal B)$ is a $\Phi_{\zeta}-$complementary pair.
In particular, $\mathcal B$ is isotropic. Moreover,
$$
\dim \mathcal Z+\dim \mathcal B=6 \quad \hbox{ and } \quad \dim \mathcal Z-\dim \mathcal B=2b^+(\mathfrak{g}).
$$

(iii) We just need to exclude the case that $\mathcal{B}\subset \Lambda_J^+(\mathfrak{g}^*)$ when $b^+(\mathfrak{g})=1$. This is impossible
because
$$
n(\mathcal{B})=2 \quad \hbox{ and } \quad n(\Lambda_J^+(\mathfrak{g}^*))=1.
$$
To see that $n(\Lambda_J^+(\mathfrak{g}^*))=1$, notice that any isotropic subspace is at least codimension 3 since $b^+(\Lambda_J^+(\mathfrak{g}^*))=3$.

(iv) Consider the intersection $\Lambda_J^+(\mathfrak{g}^*)\cap \mathcal{Z}$.
First of all, by (i),
$$
b^+(\mathfrak{g})+1\leq \dim(\Lambda_J^+(\mathfrak{g}^*)\cap \mathcal{Z}) \leq \dim (\Lambda_J^+(\mathfrak{g}^*)=4.
$$
So $\dim(\Lambda_J^+(\mathfrak{g}^*)\cap \mathcal{Z})=4$ when $b^+(\mathfrak{g})=3$.

Suppose $b^+(\mathfrak{g})=2$. Then
$\Lambda^+_J(\mathfrak{g}^*)$ cannot be contained in $ \mathcal{Z}$ since
$$b^+(\Lambda^+_J(\mathfrak{g}^*))=3 \quad \hbox{ and } \quad b^+(\mathcal{Z})=2.$$
Thus $\Lambda^+_J(\mathfrak{g}^*)\cap \mathcal{Z}$ always has dimension 3.

Suppose $b^+(\mathfrak{g})=1$. Let $L^+$ be a 3-dimensional positive definite subspace of $\Lambda^+_J(\mathfrak{g}^*)$.
We claim that $\Lambda^+_J(\mathfrak{g}^*)\cap \mathcal{Z}$ cannot have dimension 3 or 4. This is because for any 3-dimensional subspace $L\subset \Lambda^+_J(\mathfrak{g}^*)$,
$$b^+(\mathcal Z)=1 \quad \hbox{ but } \quad b^+(\Lambda^+_J(\mathfrak{g}^*))\geq b^+(L\cap L^+)=2.$$

(v) This is clear since $\mathcal B$ is isotropic and $\Lambda^-_J(\mathfrak{g}^*)$ is positive definite.

(vi) By (v) and \eqref{pm-quot-desc} we have $h_J^-=\dim (\Lambda^-_J(\mathfrak{g}^*)\cap \mathcal Z)$. Since $\Lambda^-_J(\mathfrak{g}^*)$ is positive definite we have
$\dim (\Lambda^-_J(\mathfrak{g}^*)\cap \mathcal Z)\leq b^+(\mathfrak{g}).$ Since $\dim(\Lambda^-_J(\mathfrak{g}^*))=2$, by (i), we have $b^+(\mathfrak{g})-1\leq \dim (\Lambda^-_J(\mathfrak{g}^*)\cap \mathcal Z)$
\end{proof}
We now characterize tameness in terms of $h_J^{\pm}$, from which Theorem \ref{tame-char} follows.
\begin{theorem} \label{tame-new} Let $\mathfrak{g}$ be a $4-$dimensional unimodular Lie algebra. The following are equivalent.

0) $J$ is tamed;

1) $\mathcal{B}\cap\Lambda^+_J(\mathfrak{g}^*)=\{ 0 \}$;\vskip.1truecm\noindent

2) $h_J^+= b^+(\mathfrak{g})+1$;
\vskip.1truecm\noindent

3) $h_J^-=b^+(\mathfrak{g})-1$;\vskip.1truecm\noindent

4) $H_J^+(\mathfrak{g})$ has signature $(1, b^+(\mathfrak{g}))$.
\end{theorem}
\begin{proof}
0)=1)
By the first part of Lemma \ref{invariant-pairing}, we obtain that
$G_{\zeta}(\Lambda_J^{+}(\mathfrak{g}^*))=\Lambda_J^{+}(\mathfrak{g})$ and by the second part of Lemma \ref{exact-boundary}, that $G_\eta(B)=\mathcal{B}$.
Therefore, $B\cap\Lambda^+_J(\mathfrak{g})=\{ 0 \}$ if and only if $\mathcal{B}\cap\Lambda^+_J(\mathfrak{g}^*)=\{ 0 \}$. By Theorem
\ref{compatible}, $J$ is tamed if
and only if $B\cap\Lambda^+_J(\mathfrak{g})=\{ 0 \}$.
Hence 0)=1) is proved.
\vskip.2truecm

0)=2) This is about the intersection $\Lambda_J^+(\mathfrak{g}^*)\cap \mathcal Z$. It
cannot contain $\mathcal{B}$ if $J$ is tamed. Thus in this case, by \eqref{pm-quot-desc},
we have $h_J^+= \Lambda_J^+(\mathfrak{g}^*)\cap \mathcal Z$, which is equal to $b^+(\mathfrak{g})+1$ by the fourth part of Lemma \ref{positive-boundary}.

Suppose $J$ is not tamed, then $\Lambda_J^+(\mathfrak{g}^*)\cap \mathcal{B}$ is 1 dimensional by the third part of Lemma \ref{positive-boundary}.
Thus $h_J^+= b^+(\mathfrak{g})$.

\vskip.2truecm

2)=3) Since $h_J^-=b^2-h_J^+$ by Theorem \ref{Nomizu-1-Lie-algebra}, we have the required inequality $h_J^-=b^+(\mathfrak{g})-1$.
\vskip.2truecm

4)=2) It suffices to show that 2)$\Rightarrow$4). This is true because 2)$\Rightarrow$3), and 3)$\Rightarrow$ $b^+(H_J^+(\mathfrak{g}))=1$.

The proof of Theorem \ref{tame-new} is complete.
\end{proof}
\vskip.2truecm

\begin{cor}\label{tame-new-cor} Let $\mathfrak{g}$ be a $4-$dimensional unimodular Lie algebra.

$\bullet$ When $b^+(\mathfrak{g})=3$, every $J$ is tamed and complex, and has $h_J^-=2$.

$\bullet$ When $b^+(\mathfrak{g})=2$, $J$ is either tamed or complex, corresponding to $h_J^-=1$ or $\Lambda_J^-(\mathfrak{g}^*)\subset \mathcal{Z}$ respectively.

$\bullet$ When $b^+(\mathfrak{g})=1$, $J$ is either tamed or $\Lambda_J^-(\mathfrak{g}^*)\cap \mathcal{Z}\ne \{0\}$.
\end{cor}

\begin{proof}
For the three bullets, only the second one requires a proof.
In this case, $b^+(\mathfrak{g})=2$, $J$ is tamed if and only if $h_J^-=1$.
If $J$ is not tamed then $h_J^-=2$, namely, $\Lambda_J^-(\mathfrak{g}^*)\subset \mathcal{Z}$.
By Lemma \ref{integrability-dim4} such a $J$ is integrable.
\end{proof}
Together with Proposition \ref{2-plane}, Corollary \ref{tame-new-cor} can be used to show that
there are non-tamed almost complex structure on any $4-$dimensional unimodular Lie algebra with $b^+(\g)\leq 2$. We will be content to illustrate this point with examples in \ref{examples}.
\subsection{Tamed and almost K\"ahler cone}
In this subsection we prove Theorem \ref{AK-tamed-cones-intro}.
We start with a description of the cones $\mathcal K_J^t$ and $\mathcal K_J^c$ in terms of $HPC_J$ in arbitrary dimension.

\begin{prop} \label{g} Let $J$ be an almost complex structure on a Lie algebra $\g$.

If $J$ is tamed, then under the (non-degenerate) pairing between
$H^2(M;\mathbb R)$ and $H_2(M;\mathbb R)$, the $J-$tamed cone $\mathcal K_J^t\subset
H^2(M;\mathbb R)$ is the interior of the dual cone of $HPC_J\subset H_2(M;\mathbb R)$.

If $J$ is almost K\"ahler,
then under the pairing between $H_J^+$ and
$H_+^J$, the $J-$compatible\emph{} cone $\mathcal K_J^c\subset H_J^{+}(\mathfrak{g})$ is the interior of the
dual cone of $HPC_J\subset H^J_{+}(M)$.
\end{prop}
\begin{proof} The proofs are similar, so we just prove the slightly harder second statement.
First of all, under the (possibly degenerate) pairing between
$H_J^{+}(\mathfrak{g})$ and $H^J_{+}(\mathfrak{g})$, $\mathcal
K_J^c$ is contained in the interior of the dual cone of $HPC_J$.

It remains to prove that if $e\in H_J^{+}(\mathfrak{g})$ is
positive on $HPC_J\backslash \{0\}$, then it is represented by a
$J-$compatible form.

Consider the isomorphism $\sigma$ in Lemma \ref{closed-exact} (v), and the element $\sigma(e)$ in
$\hom (\frac{\pi_{+}Z}{\pi_{+}B}, \R)$, which pulls back to a functional $L$ on
${\pi_{+}Z}$ vanishing on
${\pi_{+}B}$. Denote the kernel hyperplane of
$L$ in ${\pi_{+}Z}$ also by $L$. By our choice
of $e$, as subsets of $\Lambda_J^+(\mathfrak{g})$, $L$ and
$PC_J\backslash \{0\}$ are disjoint.

By Lemma \ref{Hahn-Banach}, we get a hyperplane $\mathcal L$ in
$\Lambda_J^+(\mathfrak{g})$ containing $L$ and disjoint from
$PC_J\backslash \{0\}$. The hyperplane $\mathcal L$ determines
a functional $\alpha$ on $\Lambda_J^+(\mathfrak{g})$, vanishing on $L\supset \pi_+B$ and being positive on
$PC_J \backslash \{0\}$.

By Lemma \ref{no-invariant-positive-boundary},
$\alpha$ is a
$J-$compatible symplectic form. Moreover, by construction we have
$[\alpha]=e$.
\end{proof}

Finally, we restate and prove Theorem \ref{AK-tamed-cones-intro}.
\begin{theorem}\label{AK-tamed-cones}
For any tamed structure $J$ on a $4-$dimensional unimodular Lie algebra,
the tamed cone $\mathcal K_J^t$ is $\mathcal K_J^c + H_J^-$,
and the compatible cone $\mathcal K_J^c$ is a connected component of
$$
\mathcal P_J=\{e\in H^+_J(\mathfrak{g})\,\,\,|\,\,\,e^2>0\},
$$
where $e^2=\bar \Phi_\zeta(e,e)$.
\end{theorem}
\begin{proof} The dimension of the compatible cone is the dimension of $H^+_J(\mathfrak{g})$.
Recall we define $HPC_J$ to be the cone in $H_+^J(\mathfrak{g})$ generated by classes of closed positive vectors,
$HPC_J=\{[u]\,\,\,|\,\,\, u\in PC_J\cap Z\}$.
By Proposition \ref{g},
the compatible cone is the $H^+_J-H_+^J$ dual of the cone $HPC_J$,
and the tamed cone is the $H^2-H_2$ dual of the cone $HPC_J$.
Furthermore, the tamed cone is the sum of the compatible cone and $H^-_J(\mathfrak{g})$.

Now we determine the $H^+_J-H_+^J$ dual of $HPC_J$.
When $J$ is tamed,
$H^+_J(\mathfrak{g})$ has signature $(1,b^+(\mathfrak{g}))$ by part 4) of Theorem \ref{tame-new}.
Any positive vector $u\in PC_J$ satisfies $\Phi_\eta(u,u) \geq 0$ by Lemma \ref{invariant-pairing} (iv).
By the light cone lemma and Lemma \ref{invariant-pairing} (i), we have
the $H^+_J-H_+^J$ dual of $HPC_J$ is a connected component of
$\mathcal P_J$.
\end{proof}
\medskip

\subsection{Examples}\label{examples}
If $J$ is a tamed almost complex structure on $\g$, then $\g$ is necessarily symplectic.
According to \cite[Theorem 9]{Chu}, a four dimensional symplectic Lie algebra is solvable. For
the classification of the four dimensional solvable Lie algebras see e.g.\cite[Theorem 1.5]{ABDO} (where product structures are investigated) or \cite[Theorem 4.1]{MS}.

The unimodular symplectic Lie algebras are
classified in \cite{Ov}. According to that list, in each case we can find a basis $\{f_1,\ldots ,f_4\}$ of $\mathfrak{g}$, such that
$$
\begin{array}{llll}
0) & \R^4, & [f_i,f_j]=0\,, & i,j=1,\ldots ,4\\
1) & \mathfrak{nil}^3\times \R, & [f_1,f_3]=f_2, &\hbox{}\\
2) & \mathfrak{nil}^4, & [f_4,f_2]=f_1, &[f_4,f_3]=f_2,\\
3) & \mathfrak{sol}^3\times \R, & [f_4,f_1]=f_1, &[f_2,f_4]=f_2,\\
4) & \mathfrak{r}^{'}_{3,0}\times \R, & [f_1,f_3]=f_2, &[f_2,f_1]=f_3.
\end{array}
$$
Denoting by $\{f^1,\ldots,f^4\}$ the dual basis of $\{f_1,\ldots,f_4\}$, we set $f^{ij}:=f^i\wedge f^j$
and so on. It follows from Proposition \ref{2-plane} and the last two bullets of Corollary \ref{tame-new-cor} that there are non-tamed $J$
when $b^+(\mathfrak{g})=1, 2$.
Applying the formula \eqref{anti-invariant-covectors} and Corollary \ref{tame-new-cor} we will produce explicitly a $2$-parameter
and a $1$-parameter family respectively of non-tamed almost complex structures on Lie algebras $\mathfrak{nil}^3\times\R$ and $\mathfrak{nil}^4$, having
$b^+(\mathfrak{g})=2$ and $b^+(\mathfrak{g})=1$ respectively.\bigskip

${\bf 1)}$ On $\mathfrak{nil}^3\times\R$, consider
the $2$-parameter family $\{J_{a,b}\}$ of almost complex structures given, with respect to the frame $\{f_1,\ldots,f_4\}$, by the following matrix
$$
\begin{pmatrix}
0 & 0 & -1 & 0\\
a & 0 & -b & -1\\
1 & 0 & 0& 0\\
-b & 1 & -a & 0
\end{pmatrix},
$$
where $a,b$ are real parameters such that $a^2+b^2\neq 0$. A direct computation, by using formula \eqref{anti-invariant-covectors}, shows that
\begin{align*}
\Lambda_{J_{a,b}}^-(\mathfrak{g}^*)&=\Span\{a(f^{34}-f^{12}) -b(f^{23}-f^{14})+
(a^2+b^2)f^{13},a(f^{23}-f^{14}) + b(f^{34}-f^{12})\}\\
&\subset \mathcal{Z}=\Span\{f^{12},f^{34},f^{14},f^{23},f^{13}\}.
\end{align*}
Since $b^+(\g^*)=2$, according to the second bullet of Corollary \ref{tame-new-cor}, $J$ cannot be tamed
by any symplectic form. Furthermore, these almost complex structures must be integrable.\bigskip

${\bf 2})$ On $\mathfrak{nil}^4$, let $\{J_t\}$, for $\vert t\vert <1$, be
represented, with respect to the frame $\{f_1,\ldots ,f_4\}$, by the following matrix,
$$
\frac{1}{1+t^2}
\begin{pmatrix}
0 & t^2-1 & 2t & 0\\
1-t^2 & 0 & 0 & -2t\\
-2t & 0 & 0& t^2-1\\
0 & 2t & 1-t^2 & 0
\end{pmatrix}.
$$
Then, by formula \eqref{anti-invariant-covectors}, we obtain, after a straightforward computation, that
$$
\Lambda_{J_{t}}^-(\mathfrak{g}^*)\ni f^{23}+f^{14}\,\in \mathcal{Z}=\Span\{f^{14},f^{23}, f^{24}, f^{34}\}.
$$
Since $b^+(\g^*)=1$, according to the third bullet of Corollary \ref{tame-new-cor}, $\{J_t\}$ cannot be tamed.\bigskip

We next offer a fun verification of Theorem \ref{AK-tamed-cones} for an explicit $J$ on
$\mathfrak{g}=\mathfrak{nil}^3\times \R$. A direct computation shows that the pair $(J,\omega)$ defined by
\begin{equation}\label{almost-complex-structure-Kodaira-Thurston}
Jf_1=f_2\,,\quad Jf_2=-f_1\,,\quad Jf_3=f_4\,,\quad Jf_4=-f_3
\end{equation}
$$
\omega =f^{12}+ f^{34}\,,
$$
gives rise to an almost K\"ahler structure on $\mathfrak{g}$ with $\zeta =f_1\wedge Jf_1\wedge f_3\wedge Jf_3$.\newline
We immediately
get that
$$
H^+_J(\mathfrak{g})=\Span_\R\{[f^{12}],[f^{34}],[f^{14}-f^{23}]\}\,,
$$
$$
H^-_J(\mathfrak{g})=\Span_\R\{[f^{14}+f^{23}]\}.
$$
Pick the connected component of
$\mathcal{P}_J$ containing $[\omega]$, and denote it by $T$.
Any class in $T$ is represented by a form $\alpha =a f^{12}+b f^{34}+
c \left(f^{14}-f^{23}\right) \in \mathcal Z_J^+$ with $\Phi_{\zeta}(\alpha, \alpha)>0$ and $\Phi_{\zeta}(\alpha,\omega) >0$, which means
$$
ab-c^2 >0\,,\quad a>0\,,\quad b >0.
$$

To show that $\alpha$ is $J-$compatible, we need to check that
$\alpha$ is positive, i.e., $\alpha(v\wedge Jv)>0$, for every $v\neq 0$.
If $$
v=Af_1+Bf_2+Cf_3+Df_4,
$$
then
$$
Jv=-Bf_1+Af_2-Df_3+Cf_4.
$$
It follows that
$$
\alpha(v,Jv)=\left(a f^{12}+b f^{34}+
c(f^{14}-f^{23})\right)(v,Jv)>0
$$
if and only if
\begin{equation}\label{a-positive}
a(A^2+B^2)+b(C^2+D^2)>-2c(AC+BD).
\end{equation}
The last condition \eqref{a-positive} is equivalent to
\begin{equation}\label{a-positiveprime}
a^2(A^2+B^2)^2+b^2(C^2+D^2)^2+2ab (A^2+B^2)(C^2+D^2)-4c^2(AC+BD)^2>0.
\end{equation}
Now,
\begin{eqnarray*}
4c^2(AC+BD)^2 &=& 4c^2(A^2C^2+B^2D^2+2ABCD)\leq 4c^2(A^2+B^2)(C^2+D^2)\\
&< & 4ab(A^2+B^2)(C^2+D^2)\\
&\leq & a^2(A^2+B^2)^2+b^2(C^2+D^2)^2+2ab (A^2+B^2)(C^2+D^2),
\end{eqnarray*}
i.e. \eqref{a-positiveprime} is proved. Hence, the compatible cone of the
almost K\"ahler $J$ on
the Lie algebra $\mathfrak{nil}^3\times \R$ is the connected component $T$ of $\mathcal P_J$.\medskip

\section{Compact quotients of $4$-dimensional Lie groups}
The main results in this paper are inspired by some recent advances on the geometry of compact $4-$dimensional
almost complex manifolds, which we offer a quick review.

For a $2n$-dimensional
compact almost complex
manifold $(M,J)$,
$J$ acts on the bundle of real
2-forms $\Lambda^2$ as an involution, by $\alpha(\cdot, \cdot)
\mapsto \alpha(J\cdot, J\cdot)$, thus we have the splitting at the bundle level $\Lambda^2=\Lambda_J^+\oplus \Lambda_J^-,$
and the splitting at the form level $\Lambda^2=\Lambda_J^+\oplus \Lambda_J^-$.
The corresponding cohomology groups $H^+_J(M)$ and $H^-_J(M)$ were introduced in \cite{LZ}.
\newline
The motivation of introducing these groups is to better understand the question of Donaldson, mentioned in the introduction
and rephrased here in terms of the tamed and almost K\"ahler cones
$$
\mathcal{K}^t_J=\left\{[\omega]\in H^2(M)\,\,\,\vert\,\,\,\omega\,\,
\hbox{\rm is}\,\, J\hbox{-\rm tamed}\right\} \quad \hbox{
and}\quad
\mathcal{K}^c_J=\left\{[\omega]\in H_J^+(M)\,\,\,\vert\,\,\,\omega\,\,
\hbox{\rm is}\,\, J\hbox{-\rm compatible}\right\}.
$$
\vskip.1truecm\noindent
{\bf Question} {\em Let $(X,J)$ be a compact $4$-dimensional almost complex manifold. Then, is it true that
$$
\mathcal{K}^t_J\neq \emptyset \iff \mathcal{K}^c_J\neq \emptyset\,?
$$
}

In \cite{DLZ}, T. Dr\u{a}ghici, the first author
and W. Zhang
proved that on any compact almost complex $4$-dimensional manifold $(M,J)$, the $2$-nd de
Rham cohomology
group decomposes as
$$
H^2(M;\R)=H^+_J(M)\oplus H^-_J(M).
$$
These groups can be viewed as a generalization of Dolbeault groups to the non-integrable case.
Indeed, in \cite[Prop.2.7]{DLZ} (see also \cite{draghici-li-zhang-1}) it
is proved that, if $J$ is a complex structure on a $4$-dimensional compact manifold, then $H^\pm_J(M)$ are identified with real Dolbeault groups, namely,
\begin{equation}\label{Dolbeault}
H^+_J(M)\simeq H^{1,1}_{\overline{\partial}}(M)\cap H^2(M;\R),\qquad
H^-_J(M)\simeq \left(H^{2,0}_{\overline{\partial}}(M)\oplus H^{0,2}_{\overline{\partial}}(M)\right)\cap H^2(M;\R)
\end{equation}

The analogues of Proposition \ref{g} and Theorem \ref{AK-tamed-cones} were first established in \cite{LZ}.
For higher dimensional results see also \cite{AT}.\newline

We end this paper with some applications of the main results to homogeneous almost complex structures.
\subsection{Nomizu-Hattori type theorem for $H^\pm_J(M)$}
The well known Nomizu-Hattori Theorem (see \cite{No}, \cite{H}) states
that if $\Gamma\backslash G$ is a compact solvmanifold of {\em completely solvable type}, i.e., the adjoint representation
of $\mathfrak{g}$ has real eigenvalues, then $H^*(\Gamma\backslash G)\cong H^*({\mathfrak g})$. \newline

We prove a Nomizu-Hattori type theorem for $H^\pm_J(M)$, where $M$ is a $4$-dimensional compact quotient of a simply-connected
Lie group $G$, more precisely
\begin{theorem}\label{Nomizu-Hattori-Theorem}
Let $M=\Gamma\backslash G$ be a compact quotient
of a simply-connected $4$-dimensional real Lie group $G$ by a uniform
discrete subgroup $\Gamma\subset G$ and let $J$ be an almost complex structure on
the Lie algebra ${\mathfrak g}$ of $G$. Assume that $H^2(\mathfrak{g})
\cong H^2(M)$. Then
$$
H^+_J(M)=H^+_{J}({\mathfrak g})\,,\quad H^-_J(M)=H^-_{J}({\mathfrak g}).
$$
\end{theorem}
\begin{proof}
In view of a result of Milnor (see \cite[Lemma 6.2]{M}), the existence of a uniform discrete subgroup of $G$ implies that the Lie algebra
$\mathfrak{g}$ of $G$ has to be unimodular. Let $J$ be the induced left-invariant almost complex structure
on $M=\Gamma\backslash G$. In view of the assumption $H^2(\mathfrak{g})
\cong H^2(M)$, of the fact that
$J$ is
$\mathcal{C}^\infty$ pure and full \cite[Theorem 2.3]{DLZ} and of Theorem \ref{Nomizu-1-Lie-algebra}, we have that
$$
H^2(\mathfrak{g})=H^+_J({\mathfrak g})\oplus H^-_J({\mathfrak g})\stackrel{\cong}\hookrightarrow
H^2(M;\R)=H^+_J(M)\oplus H^-_J(M)
$$
Noticing that there are natural homomorphisms,
$$
H^+_J({\mathfrak g})\longrightarrow H^+_J(M)\,,\quad H^-_{J}({\mathfrak g})\longrightarrow H^-_J(M),
$$
the conclusion follows.
\end{proof}
\begin{rem}
For the Dolbeault cohomology, in \cite{CF} (see also \cite{CFP}) a Nomizu-Hattori theorem is proved: more precisely, if
$M=\Gamma\backslash G$ is a compact nilmanifold endowed with an {\em Abelian} left-invariant complex structure $J$, namely, $J$ is a complex
structure on $\mathfrak{g}$ such that $[Ju,Jv]=[u,v]$, for every $u,v\in\mathfrak{g}$, then
$H^{p,q}_{\overline{\partial}}(M)\simeq H^{p,q}_{\overline{\partial}}(\mathfrak{g}_\C)$.
\newline
Let $J$ be a complex structure on a $4$-dimensional Lie algebra $\mathfrak{g}$. Then arguing as in \cite[Prop.2.7]{DLZ}, it can
be showed that the same decompositions
\eqref{Dolbeault} hold, for real Dolbeault groups of $\mathfrak{g}$. Therefore, in dimension four, under
the weaker assumption that $G$ is completely solvable, a consequence of Theorem \ref{Nomizu-Hattori-Theorem} is a Nomizu-Hattori theorem for
real Dolbeault groups, namely,
$$
H^{1,1}_{\overline{\partial}}(M)\cap H^2(M;\R)\simeq H^{1,1}_{\overline{\partial}}(\mathfrak{g})\cap H^2(\mathfrak{g})
$$
$$
\left(H^{2,0}_{\overline{\partial}}(M)\oplus H^{0,2}_{\overline{\partial}}(M)\right)\cap H^2(M;\R)\simeq
\left(H^{2,0}_{\overline{\partial}}(\mathfrak{g})\oplus H^{0,2}_{\overline{\partial}}(\mathfrak{g})\right)\cap H^2(\mathfrak{g}).
$$
\end{rem}
\vspace{0.2cm}

\subsection{Tamed and almost K\"ahler structures on quotients}
Finally, we note that there are analogues of Theorems \ref{tamed-and-compatible-intro}, \ref{tame-char}, and \ref{AK-tamed-cones-intro}.
\begin{theorem} \label{AK-and-Tamed-Lie-Groups}
Let $M=\Gamma\backslash G$ be a compact quotient of a $4$-dimensional simply-connected Lie group $G$,
where $\Gamma\subset G$ is a uniform discrete subgroup. Let $J$ be a
left-invariant almost complex structure on $M$. Then the following are equivalent:

0) $J$ is tamed by a symplectic form $\omega$ on $M$ (not necessarily left-invariant).

1) $J$ is tamed by an left-invariant symplectic form $\omega$ on $M$.

2) There are no non-trivial $J$-invariant exact $2$-forms on $M$.

3) $M$ carries a left-invariant symplectic
form compatible with $J$.

4) $M$ carries a symplectic
form compatible with $J$ (not necessarily left-invariant).

5) $h_J^-=b^+(M)-1$.

\noindent Furthermore, if $J$ is tamed, the $J-$compatible cone is a connected component of $\mathcal P_J=\{e\in H^+_J(M)\,\,\,\vert\,\,\, e^2>0\}$.
\end{theorem}

As already recalled, the Lie algebra
$\mathfrak{g}$ of $G$ is unimodular.
Assume that $\omega$ is a non left-invariant
symplectic structure on $M$ taming $J$ (compatible with $J$). Then by the
same process as in \cite{FGr}, we may construct a left-invariant symplectic structure $\tilde{\omega}$ on
$M$, taming $J$ (compatible with $J$).

In view of these facts, all the statements follow from
the corresponding results for 4-dimensional unimodular Lie algebras, specifically, Theorems \ref{tame-new}, \ref{compatible}, \ref{AK-tamed-cones}.

It is amusing to notice that, for a compact $4$-manifold, tameness of an integrable almost complex structure
 (\cite{HL}, see also \cite{DLZ2}) and tameness of a homogeneous almost complex structure can be both characterized by $h_J^-=b^+-1$.\vskip.2truecm

We note that there is an alternative argument for the equivalence between 1) and 3) in \cite{DLZ2}.
We also note that the K\"ahler cone is determined in terms of the curve cone by Lamari, Buchdahl in \cite{Lamari} and \cite{Bu} for a K\"ahler surface,
by Demailly and Paun in \cite{DP} for an arbitrary K\"ahler manifold, which generalizes the famous Nakai-Moizheson ampleness criterion in projective geometry.
In the almost K\"ahler case in dimension 4, such a generalization is contained in Taubes' \cite{T1} for a generic $J$ when $b^+=1$, and for several
geometrically interesting families
of almost complex structures on rational manifolds in \cite{LZ2}.


\begin{thebibliography}{12}
\bibitem{ABDO} A. Andrada, M. L. Barberis, I. Dotti, and G. P. Ovando, Product structures on four dimensional solvable Lie algebras, {\em Homology Homotopy Appl.} {\bf 7} (2005), 9--37.
\bibitem{AT} D. Angella, A. Tomassini, On cohomological decomposition of almost complex
manifolds and deformations, {\em J. Symplectic Geom.} {\bf 9} (2011), 1--26.
\bibitem{Bu} N. Buchdahl, A Nakai-Moishezon criterion for non-K\"ahler surfaces,
{\em Ann. Inst. Fourier} {\bf 50} (2000), 1533--1538.
\bibitem{Chu} B. Y. Chu, Symplectic homogeneous spaces, {\em Trans. Amer. Math. Soc.}
{\bf 197} (1974), 145--159.
\bibitem{CF} S. Console, A. Fino, Dolbeault cohomology of compact nilmanifolds, {\em Transform. Groups} {\bf 6} (2001), 111--124.
\bibitem{CFP} S. Console, A. Fino, Y.S. Poon, Stability of abelian complex structures,
{\em Internt. J. Math.} {\bf 17} (2006), no. 4, 401--416.
\bibitem{DP} J.-P. Demailly, M. Paun, Numerical characterizations of the K\"ahler cone of a
compact K\"ahler manifold,
{\em Ann. of Math. (2)} {\bf 159} (2004), 1247--1274.
\bibitem{D} S.~K. Donaldson, {\em Two-forms on four-manifolds and elliptic equations},
Inspired by {S}. {S}. {C}hern, Nankai Tracts Math., vol.~11, World Sci. Publ., Hackensack,
NJ, 2006, pp.~153--172.
\bibitem{DLZ} T. Dr\v{a}ghici, T.-J. Li, W. Zhang, Symplectic forms and cohomology decomposition
of almost complex $4$-manifolds, {\em Intern. Math. Res. Not.} {\bf 2010} (2010), 1--17.
\bibitem{draghici-li-zhang-1} T. Dr\v{a}ghici, T.-J. Li, W. Zhang, On the $J$-anti-invariant cohomology
of almost complex $4$-manifolds, to appear in {\em Q. J. Math.}.
\bibitem{DLZ2} T. Dr\v{a}ghici, T.-J. Li, W. Zhang, Geometry of tamed almost complex structures on
4-dimensional manifolds, to appear in {\em ICCM2010 Proceedings}.
\bibitem{EF} N. Enrietti, A. Fino, Special Hermitian metrics and Lie groups, {	\tt arXiv:1104.1612v1}.
\bibitem{FT} A. Fino, A. Tomassini, On some cohomological properties of almost complex
manifolds, {\em J. Geom. Anal.} {\bf 20}
(2010), 107--131.
\bibitem{FGr} A. Fino, G. Grantcharov, Properties of manifolds with skew-symmetric torsion
and special holonomy, {\em Adv. Math.} {\bf 189} (2004), 439--450.
\bibitem{GHR} S.~J. Gates, C.~M. Hull, M. Ro\v cek, Twisted multiplets and new supersymmetric nonlinear sigma models,
{\em Nuc. Phys. B} {\bf 248} (1984) 157--186.
\bibitem{G} M. Gualtieri, Generalized complex geometry, {\em Ann. of Math. (2)} {\bf 174} (2011), 75--123.
\bibitem{HL} R. Harvey, H. Lawson, An intrinsic charactherization of K\"ahler manifolds,
{\em Invent. Math.} {\bf 74} (1983), 169--198.
\bibitem{H} A. Hattori, Spectral sequence in the de Rham cohomology of
fibre bundles, {\em J. Fac. Sci. Univ. Tokyo.} {\bf 8} (1960), 289--331.
\bibitem{K} J.-L. Koszul, Homologie et cohomologie des alg\'ebres de Lie, {\em Bull. Soc. Math. France} {\bf 78} (1950), 65--127.
\bibitem{Lamari} A. Lamari, Le c\^{o}ne k\"ahl\'{e}rien d'une surface,
{\em J. Math. Pures Appl.} (9) {\bf 78} (1999), 249--263.
\bibitem{L} T.-J. Li, Symplectic Calabi-Yau surfaces,
{\em Handbook of Geometrical Analysis}, No. 3, 231–-356, Adv. Lect. Math. (ALM), 14, Int. Press, Somerville, MA, 2010.
\bibitem{LZ} T.-J. Li, W. Zhang, Comparing tamed and compatible symplectic cones and cohomological
properties of almost
complex manifolds, {\em Comm. Anal. Geom.} {\bf 17} (2009), 651--683.
\bibitem{LZ2} T.-J. Li, W. Zhang, J-symplectic cones of rational four manifolds, \texttt{preprint}.
\bibitem{MS} T. B. Madsen, A. Swann,  Invariant strong KT geometry on four-dimensional solvable Lie groups, {\em J. Lie Theory} {\bf 21} (2011), 55--70.
\bibitem{M} J. Milnor, Curvature of left invariant metrics on Lie groups, {\em Advances in Math.} {\bf 21} (1976), no. 3, 293329.
\bibitem{No} K. Nomizu, On the cohomology of compact homogeneous spaces of nilpotent
Lie groups, {\em Ann. of Math.}
{\bf 59} (1954), 531--538.
\bibitem{Ov} G. Ovando, Four Dimensional Symplectic Lie Algebras, {\em Beitr\"age Algebra Geom.}
{\bf 47}
(2006), 419--434.
\bibitem{S} A. Strominger, Superstrings with torsion, {\em Nuclear Phys. B}
{\bf 274} (1986) 253--284.
\bibitem{sullivan} D. Sullivan, Cycles for the dynamical study of foliated manifolds and complex manifolds,
{\em Invent. Math.}
{\bf 36} (1976), 225--255.
\bibitem{T1} C.H.Taubes, Tamed to compatible: symplectic forms via moduli space integration, {\em J. Symplectic Geom.} {\bf 9} (2011), 161--250.
\end{thebibliography}
\end{document}